\documentclass[10pt,reqno]{amsart} 

\usepackage{amssymb,latexsym,  
amscd, 
amsthm, 
graphicx, epsfig
}
\usepackage{cite} 

\usepackage[height=190mm,width=130mm]{geometry} 

\theoremstyle{plain}

\newtheorem{lemma}{Lemma}
\newtheorem{corollary}{Corollary}
\newtheorem{proposition}{Proposition}

\theoremstyle{definition}
\newtheorem{definition}{Definition}

\theoremstyle{remark}
\newtheorem{remark}{Remark}

\newcommand{\Z}{\mathbb{Z}}
\newcommand{\W}{\mathcal{W}}

\newcommand{\Oo}{\mathcal{O}}
\newcommand{\C}{\mathbb{C}}

\numberwithin{equation}{section} 

\newcommand{\nn}{\nonumber \\}

 \newcommand{\res}{\mbox{\rm Res}}

\renewcommand{\hom}{\mbox{\rm Hom}}

\newcommand{\wt}{\mbox{\rm wt}\ }

\newcommand{\N}{\mathbb{N}}

\newcommand{\F}{\mathcal{F}}

\newcommand{\one}{\mathbf{1}}

\begin{document}
\title[The product on $\W$-spaces of rational forms]     
{The product on $\W$-spaces of rational forms}  
\author{A. Zuevsky}
\address{Institute of Mathematics \\ Czech Academy of Sciences, \\ Zitna 25, 
11567 \\ Prague\\ Czech Republic}     

\email{zuevsky@yahoo.com}
\begin{abstract}
 Let $V$ be a quasi-conformal grading-restricted vertex algebra, $W$ be its module,  
 and $\W_{z_1, \ldots, z_n}$ be the space 
of rational differential forms with complex parameters $(z_1, \ldots, z_n)$
 for $n \ge 0$.  
Using geometric interpretation in terms of two Riemann spheres sewing 
we define a product of elements of two spaces $\W_{x_1, \ldots, x_k}$ 
and $\W_{y_1, \ldots, y_n}$, and study its properties. 
A product is introduced also for elements of 
 two spaces $C^k_m(V, \W)$ $\times$ $C^n_{m'}(V, \W)$ $\to$ $C^{k+n}_{m+m'}(V, \W)$  
of the corresponding chain complex of rational differential forms 
invariant with respect to transformations of complex parameters.  
\end{abstract}
\vskip12pt  

\maketitle
\section{Conflict of Interest}
The author states that: 

1.) The paper does not contain any potential conflicts of interests. 
\section{Data availability statement}
The author confirms  that: 

1.) All data generated or analyzed during this study are included in this published article. 

2.)   Data sharing not applicable to this article as no datasets were generated or analysed during the current study.
\section{Introduction}
\label{introduction}
The problem of defining a product on the space $\W_{z_1, \ldots, z_n}$ (or $\W$-spaces) 
of rational differential forms 
(and, in particular, on $C^n_m(V, \W)$-spaces of corresponding chain complex)   
is very important for cohomology theory of vertex algebras, continual Lie algebras, 
 theory of integrable models,  
 as well as for further applications to 
cohomology of smooth manifolds.  
A cohomology theory for grading-restricted vertex algebras was introduced in \cite{Huang}.  
Vertex algebras, generalizations of ordinary Lie algebras, 
are essential in conformal field theory   
\cite{FMS}, and it is a rapidly developing 
field of studies. 
Algebraic nature of methods applied in this field helps to understand and compute  
the structure of vertex algebra characters \cite{FMS, H2, Zhu, FHL, BZF}.  
On the other hand, the geometric side of a vertex algebra characters 
is in association of their formal parameters with   
local coordinates on a complex variety.  
Depending on geometry of a manifold, 
one can obtain various consequences for a vertex algebra 
and its space 
of characters.  Vice-versa, one can study geometric property of a manifold by using 
algebraic nature of a vertex algebra attached. 
 
For purposes of cohomology invariants construction for vertex algebras,  
it is important to define products of elements of chain-cochain complex spaces.   
The natural and well-developed theory of vertex operator algebra correlation functions 
\cite{MT1, MT2, MT3, TZ, TZ2, HT, MT4, GT1, GT2}  
considered on higher genus Riemann surfaces \cite{Y} and  
constructed from lower genus Riemann surfaces  
proved its convenience for computations, and 
served as a motivation for our construction of a product of rational forms. 
In that direction, an extremely difficult question of 
composability with vertex operators occurs.    
For the cohomology theory of vertex algebras, 
one has to assume that the chain-cochains 
are composable with vertex operators which assumes convergence.   
Especially, 
 when we want to compute  
 cohomology of a vertex algebra, we  
have to deal with the 
convergence problem first.
In the case of grading-restricted vertex algebras,  
 the difficulty 
is that chain-cochains are not represented by 
vertex or intertwining operators.  
The techniques for vertex operators 
or intertwining operators in general do not work.
The aim of this paper is to develop such new techniques.  

For products of spaces of chain-cochains, 
we propose to involve the geometric procedure \cite{Y} of sewing of Riemann surfaces 
as auxiliary model spaces in a geometric interpretation of algebraic products of spaces 
associated to vertex algebras.    
Similar to various 
other structures in the theory of vertex
operator algebras, this is  
not a usual associative product. 
The product that occurs is parametrized by a nonzero complex number $\epsilon$ 
identified to the complex parameter of 
the sewing procedure involved.  
More generally, the product is constructed from two Riemann 
spheres with a collection of marked points with  
 local coordinates vanishing at these points.  
The same scheme works, 
 for example, for tensor products of modules which are in fact parametrized by
 such geometric objects.
 Because of this, the existence of such products
 involves the convergence. 
In addition to that, 
 a vertex operator algebra must satisfy some conditions in order for such
convergence to hold. 

In this paper we introduce the product of elements of $\W$-spaces of rational differential 
forms  
for a grading-restricted vertex algebra.    
For the construction of complexes (cf. Section \ref{application})  
we make use of maps from tensor powers of $V$ to the space 
$\W_{z_1, \dots, z_n}$ 
 to define cochains in vertex algebra cohomology theory.
 For that purpose, in particular, to define coboundary operators,  
we have to compose chain-cochains with vertex operators. However,  
 the images of vertex operator maps in general do not belong to 
algebras or their modules.
 Such objects 
belong to corresponding algebraic completions which constitute 
  one of the most subtle features of the theory of vertex algebras.
 Because of this, we might not be able to compose vertex operators directly. 
 In order to overcome this problem, 
 one first writes a series by projecting 
 an element of the algebraic completion 
of an algebra or a module to its homogeneous components.  
 Then we compose these homogeneous components with vertex operators,  
and take formal sums.
 If such formal sums are absolutely convergent, then these operators 
can be composed and can be used in constructions. 
 
The plan of the paper is the following. 
In Section \ref{valued} we recall the definition of the space of $\W$-valued rational forms
for grading-restricted vertex algebras, and remind properties of their elements. 
In Section \ref{product} we introduce a product for elements 
of two $\W_{z_1, \ldots, z_n}$-spaces.   
In Section \ref{properties} we study properties of the resulting product.   
In Section \ref{application} we recall the definition and properties of    
spaces $C^n_m(V, \W)$ of 
the chain-cochain complex for a grading-restricted vertex algebra.  
In Section \ref{productc} we define the product for elements of   
 $C^n_m(V, \W)$-spaces and study its properties. 
In Section \ref{example} we consider the particular case of a short exceptional complex 
associated to certain $C^n_m(V, \W)$ subspaces.  
In Subsection \ref{appi} we provide a specific example of application of our product 
to cohomology invariants for vertex algebras.  
In Appendixes we provide the material needed for 
the construction of the product for $\W$-spaces.  
In Appendix \ref{grading} we recall the notion of a quasi-conformal
 grading-restricted vertex algebra.  
 In Appendix \ref{sphere} we describe the geometric procedure 
of forming a Riemann sphere by sewing  
two initial Riemann spheres.
Finally, Appendix \ref{proof} contains a proof of Proposition \ref{nezc}. 
\section{Spaces of $\W$-valued rational forms}
\label{valued}
\subsection{The space $\W$ of rational forms} 
\label{valuedef}
Part of notions and notations in this Subsection originates from \cite{Huang}.  
 We define the configuration spaces: 
\[
F_n\C=\{(z_1, \dots, z_n)\in \C^n \;|\; z_i\ne z_j, i\ne j\}, 
\]
for $n\in \Z_+$.
Recall definitions and properties given in Appendix \ref{grading}.  
Let $V$ be a grading-restricted 
vertex algebra, and $W$ a grading-restricted generalized $V$-module.   
By $\overline{W}$ we denote the algebraic completion of $W$, 
\[
\overline{W}=\prod_{n\in \mathbb C}W_{(n)}=(W')^*. 
\]
Let $w'\in W'$ be an arbitrary element of $W'$ dual to $W$ with respect to the 
canonical pairing $\langle . , . \rangle$ with the dual space of $W$. 
\begin{definition}
\label{valuedrational}
 A $\overline{W}$-valued rational function $f$ in $(z_1, \dots, z_n)$  
with the only possible poles at 
$z_i=z_j$, $i\ne j$,  
is a map 
\begin{eqnarray*}
 f:F_n\C &\to& \overline{W},   
\\
 (z_1, \dots, z_n) &\mapsto& f(z_1, \dots, z_n),    
\end{eqnarray*} 
such that for any $w'\in W'$,  
\begin{equation}
\label{deff}
R(z_1, \ldots, z_n)= R\left(\langle w', f(z_1, \dots, z_n) \rangle \right), 
\end{equation} 
is a rational function in $(z_1, \dots, z_n)$   
with the only possible poles at
$z_i=z_j$, $i\ne j$.  
In this paper, such a map is called $\overline{W}$-valued rational function 
in $(z_1, \dots, z_n)$ with possible other poles. 
 The space of $\overline{W}$-valued rational functions is denoted by 
$\overline{W}_{z_1, \dots, z_n}$. 
\end{definition} 
 Here $R(.)$ denotes the following.    
Namely, if a meromorphic function $f(z_1, \dots, z_n)$ on a region in $\C^n$   
can be analytically extended to a rational function in $(z_1, \dots, z_n)$, 
then the notation $R(f(z_1, \dots, z_n))$ is used to denote such rational function.  
Note that a set of a grading-restricted vertex algebra elements $(v_1, \ldots, v_n)$ 
associated with the  
corresponding $(z_1, \ldots, z_n)$ 
plays the role of non-commutative parameters for a function $f$ in \eqref{deff}.   
Recall (Appendix \ref{grading}) the definition of 
a quasi-conformal grading-restricted vertex algebra $V$.   
Let us introduce the definition of a $\W_{z_1, \ldots, z_n}$-space: 
\begin{definition}
\label{wspace}
We define the space $\W_{z_1, \dots, z_n}$ of 
 $\overline{W}_{z_1, \dots, z_n}$-valued rational forms    
\begin{eqnarray}
\label{bomba}
&& 
\F \left(v_1, z_1; \ldots; 
  v_n, z_n\right)
\nn
 && \qquad 
= \Phi \left(dz_1^{{\rm \wt}(v_1)} \otimes v_{1}, z_1; \ldots;
 dz_n^{{\rm \wt}(v_n)} \otimes  v_n, z_n\right),  
\end{eqnarray}
formed by all $\overline{W}$-valued rational functions $\Phi$
with each vertex algebra element entry $v_i$, $1 \le i \le n$  
of a grading-restricted quasi-conformal vertex algebra $V$   
 tensored with power $\wt(v_i)$-differential of 
the corresponding formal parameter $z_i$. 
We call $\W$ the space of $\W_{z_1, \dots, z_n}$ for $n=0$.    
\end{definition}
Recall definitions and properties of Subsection \ref{perdozo} of Appendix \ref{grading}. 
Let us denote $\Oo^{(n)}$ the space of formal Taylor series in $n$ variables. 
In Appendix \ref{proof} we give a proof of the following 
\begin{proposition}
\label{ndimwinv}
For primary vectors of a quasi-conformal grading-restricted vertex algebra $V$,  
 the form \eqref{bomba} 
is invariant with respect to elements 
\[
\left(\rho_1(z_1, \ldots, z_n), \ldots, \rho_n(z_1, \ldots, z_n)\right), 
\]
 of the group 
${\rm Aut}_{z_1, \ldots, z_n}\Oo^{(n)}$,   
 i.e., 
under  
the changes 
\[
z_i \mapsto z_i'=\rho_i(z_1, \ldots, z_n),  
\]
 of formal parameters  
$(z_1, \dots, z_n)$.  
\end{proposition} 
\subsection{Properties of rational functions for $\W$-valued elements}  
\label{properties} 
Let $V$ be a grading-restricted vertex algebra and $W$ a grading-restricted 
generalized $V$-module  
(cf. Appendix \ref{grading}).  
Let us give here modifications of definitions and facts 
about matrix elements for a grading-restricted vertex algebra.  
\begin{definition}
\label{wspace}
For $n\in \Z_+$,  
a map 
\[
\F(v_1, z_1;  \ldots ; v_n,  z_n) 
\in V^{\otimes n}\to 
\W_{z_1, \dots, z_n},  
\]
 is said to have
the $L_V(-1)$-derivative property if
\begin{equation}
\label{lder1}
(i) \qquad 
\partial_{z_i} \F (v_1, z_1;  \ldots ; v_n,  z_n)  
=  
\F(v_1, z_1;  \ldots; L_V(-1)v_i, z_i; \ldots ; v_n,  z_n),   
\end{equation}
for $i=1, \dots, n$, $(v_1, \dots, v_n) \in V$, $w'\in W'$,  
and  
\begin{eqnarray}
\label{lder2}
(ii) \qquad \sum\limits_{i=1}^n\partial_{z_i} 
\F(v_1, z_1;  \ldots ; v_n,  z_n)=  
 L_W(-1).\F(v_1, z_1;  \ldots ; v_n,  z_n).  
\end{eqnarray} 
\end{definition}
In \cite{Huang} we find the following 
\begin{proposition}
Let $\F$ be a map having the $L_W(-1)$-derivative property. 
Then for $(v_{1}, \dots, v_{n})\in V$,   
 $(z_1, \dots, z_n)\in F_n\C$, $z\in \C$ such that 
$(z_1+z, \dots,  z_n+z)\in F_n\C$,  
\begin{eqnarray}
\label{ldirdir}
e^{zL_W(-1)}  
\F \left(v_1, z_1; \ldots; v_n, z_n \right)  
= 
\F(v_1, z_1+z ; \ldots; v_n,  z_n+z),    
\end{eqnarray}  
and $1\le i\le n$ such that
\[
(z_1, \dots, z_{i-1}, z_i+z, z_{i+1}, \dots, z_n)\in F_n\C, 
\]
the power series expansion of 
\begin{equation}
\label{expansion-fn} 
\F(v_1, z_1;  \ldots; v_{i-1}, z_{i-1};  v_i, z_i+z; v_{v+1}, z_{i+1}; \ldots  v_n, z_n),   
\end{equation}
in $z$ is equal to the power series
\begin{equation}
\label{power-series}  
\F\left(v_1 z_1; \ldots;  v_{i-1}, z_{i-1}; e^{zL_V(-1)}v_i, z_i;   
 v_{i+1}, z_{i+1};  \ldots; v_n, z_n\right), 
\end{equation}
in $z$.
In particular, the power series \eqref{power-series} in $z$ is absolutely convergent
to \eqref{expansion-fn} in the disk $|z|<\min_{i\ne j}\{|z_i-z_j|\}$.   
\end{proposition}
One states   
\begin{definition}
A map 
\[
\F: V^{\otimes n} \to \W_{z_1, \dots, z_n} 
\]
 has the $L_W{(0)}$-conjugation property if for $(v_1, \dots, v_n) \in V$,  
$(z_1, \dots, z_n)\in F_n\C$, and $z\in \C^\times$, such that  
$(zz_1, \dots, zz_n)\in F_n\C$, 
\begin{eqnarray}
\label{loconj} 
 z^{L_W(0)}   
\F \left(v_1, z_1; \ldots; v_n, z_n \right) 
=
 \F\left(z^{ L_V{(0)} } v_1, zz_1;  \ldots ;  z^{L_V{(0)} } v_n,  zz_n\right). 
\end{eqnarray} 
\end{definition}
Let $S_n$ be the permutation group.  
 One defines the action of $S_n$ on the space $\hom(V^{\otimes n},  
\W_{z_1, \ldots, z_n})$ of maps from  
$V^{\otimes n}$ to $\W_{z_1, \dots, z_n}$ by    
\begin{equation}
\label{sigmaction}
\sigma(\F)(v_1, z_1; \ldots; v_n, z_n)   
=\F (v_{\sigma(1)}, z_{\sigma(1)};  \ldots v_{\sigma(n)}, z_{\sigma(n)}), 
\end{equation}  
for $\sigma\in S_n$, and $(v_1, \ldots, v_n)\in V$. 
We will use the notation $\sigma_{i_1, \ldots, i_n}\in S_n$, to denote 
the permutation given by $\sigma_{i_1, \ldots, i_n}(j)=i_j$,    
for $j=1, \dots, n$.

Finally, the following result was proved in \cite{FHL}: 
\begin{proposition}
\label{n-comm}
For $(v_1, \dots, v_n)\in V$, $w\in W$ and 
$w'\in W'$, 
\[
\langle w', Y_W(v_1, z_1)\ldots Y_W(v_n, z_n)w\rangle,  
\]
is absolutely convergent in the region $|z_1|>\ldots >|z_n|>0$  
to a rational function 
\[
R(\langle w', Y_W(v_1, z_1)\ldots Y_W(v_n, z_n)w\rangle),  
\]
in $(z_1, \dots, z_n)$ with the only possible poles at $z_i=z_j$, $i\ne j$, 
and $z_i=0$.  
The following commutativity holds: for 
$\sigma\in S_n$, 
\begin{eqnarray*}
\lefteqn{R(\langle w', Y_W(v_1, z_1)\ldots Y_W(v_n, z_n)w\rangle)}\nn 
&&=R(\langle w', Y_W(v_{\sigma(1)}, z_{\sigma(1)})\ldots 
Y_W(v_{\sigma(n)}, z_{\sigma(n)})w\rangle).
\end{eqnarray*}
\end{proposition}
\section{Product of spaces of $\W$-valued forms}
\label{product}
\subsection{Motivation and geometric interpretation} 
The structure of $\W_{z_1, \ldots, z_n}$-spaces 
is quite complicated and it is difficult to introduce algebraically a product 
of its elements.  
In order to define an appropriate product of two elements of 
$\W_{z_1, \ldots, z_n}$-spaces we first have to interpret  
them geometrically. 
Basically, a $\W_{z_1, \ldots, z_n}$-space must be associated with 
a certain model space, the algebraic $\W$-language should be  
transferred to a geometric one, two model spaces should be "related" appropriately, and,
 finally, a product should be 
defined. 

For two  
 $\W_{x_1, \ldots, x_k}$- and 
$\W_{y_1, \ldots,  y_n}$-spaces we first associate formal complex parameters   
in the sets 
$(x_1, \ldots, x_k)$ and $(y_1, \ldots, y_n)$  
to parameters of two auxiliary   
spaces. 
 Then we describe a geometric procedure to 
form a resulting model space 
by combining two original model spaces.  
Formal parameters of 
$\W_{z_1, \ldots, z_{k+n}}$ should be then identified with  
parameters of the resulting space.  

Note that 
according to our assumption, $(x_1, \ldots, x_k) \in F_k\C$, 
and $(y_1, \ldots, y_n) \in  F_n\C$.  
As it follows from definition of the configuration space $F_n\C$ given in Subsection 
\ref{valuedef}, 
 in the case of coincidence of two 
formal parameters they are excluded from $F_n\C$.  
In general, it may happen that some number $r$ 
of formal parameters of $\W_{x_1, \ldots, x_k}$ coincides with some  
$r$ formal parameters of $\W_{y_1, \ldots, y_n}$. Thus,   
we require that the set of formal parameters $(z_1, \ldots, z_{k+n-r})$ 
for the resulting model space 
 would belong to $F_{k+n-r}\C$. 
This leads to the fall off of the total number 
of formal parameters for the resulting model space  
$\W_{z_1, \ldots, z_{k+n-r}}$. 
In what follows we consider the case when 
all formal parameters $(x_1, \ldots, x_k)$ differ  
from formal parameters of $(y_1, \ldots, y_n)$.   
This singular case can then be treated
 similar to the ordinary one in lower dimension.  
\subsection{Definition of the 
product for $\W$-valued rational forms} 
\label{genus_two_n-point}
Recall the definition \eqref{wprop} of the intertwining operator 
$Y^W_{WV}$ given in Appendix \ref{grading}. 
We then formulate  
\begin{definition}
\label{duplodef}
For a quasi-conformal module $W$ of a grading-restricted vertex algebra $V$,  
 a set of quasi-primary $V$-elements $(v_1, \ldots,  
v_n )$, $(v'_1, \ldots,  v'_n )\in V$, 
 $\F (v_1, x_1$; $\ldots$ ; $v_k, x_k)$    
$\in \W_{x_1,  \ldots ,  x_k}$,       
 $\F ( v'_1, y_1; \ldots ; v'_n, y_n)$, $\in \W_{y_1,  \ldots ,  y_n}$, 
 introduce the $\epsilon$-product for $\epsilon=\zeta_1 \zeta_2$, for $|\zeta_a|>0$,
 $a=1$, $2$, 
\begin{equation}
\label{gendef}
\cdot_\epsilon: \W_{x_1, \ldots, x_k} \times \W_{y_1, \ldots, y_n} 
\to \W_{x_1, \ldots, x_k; y_1, \ldots, y_n},       
\end{equation}  
for $(x_1, \ldots, x_k; y_1, \ldots, y_n ) \in F_{k+n}\C$.     
For arbitrary $w' \in W'$, the product is associated to the form
\begin{eqnarray}
\label{Z2n_pt_eps1q1}
&& \mathcal R(x_1, \ldots, x_k; y_1, \ldots, y_n; \epsilon, \zeta_1, \zeta_2)  
\nn
 & &  =  \sum_{l\in \mathbb{Z} } \epsilon^l \sum_{u\in V_l}    
\langle w', Y^W_{WV}\left(  
\F (v_1, x_1;  \ldots; v_k, x_k), \zeta_1 \right) \; u \rangle  
\nn
& &
\qquad   \qquad 
\langle w', Y^W_{WV}\left(\F 
(v'_1, y_1; \ldots; v'_n, y_n),  \zeta_2 \right) \overline{u} \rangle,     
\end{eqnarray}
via \eqref{deff}, 
parametrized by  
$\zeta_a \in \C$, $|\zeta_a|>0$, $a=1$, $2$, 
for all pairs of matrix elements  
$\langle w'$, $Y^W_{WV}\left(  
\F (v_1, x_1;  \ldots; v_k, x_k), \zeta_1 \right) \; u \rangle$, 
$\langle w'$, $Y^W_{WV}\left(\F 
(v'_1, y_1; \ldots; v'_n, y_n),  \zeta_2 \right) \overline{u} \rangle$ 
such that \eqref{Z2n_pt_eps1q1} converges in $\epsilon$. 
The sum 
is taken over any $V_l$-basis $\{u\}$,  
where $\overline{u}$ is the dual of $u$ with respect to the 
canonical pairing $\langle . , . \rangle_\lambda$ \eqref{eq: inv bil form}  
 with the dual space of $V$, (see Appendix \ref{grading}).   
\end{definition}
By the standard reasoning \cite{FHL, Zhu},   
 \eqref{Z2n_pt_eps1q1} does not depend on the choice of a basis of $u \in V_l$, $l \in \Z$.   
In the case when multiplied forms $\F$ do not contain $V$-elements, 
i.e., for $\Phi$, $\Psi \in \W$,  
  \eqref{Z2n_pt_eps1q1} defines the product $\Phi \cdot_\epsilon \Psi$ associated to 
 a rational function: 
\begin{eqnarray}
\label{Z2_part}
{\mathcal R(\epsilon)}= \sum_{l \in \Z }  \epsilon^l  
 \sum_{u\in V_l } 
\langle w', Y^W_{WV}\left(  
\Phi, \zeta_1\right) \; u \rangle 
\langle w', Y^W_{WV}\left(
\Psi,   
\zeta_2 \right) \; \overline{u} \rangle,  
\end{eqnarray}  
which defines $\F(\epsilon) \in \W$ 
via $\mathcal R(\epsilon)=\langle w', \F(\epsilon)\rangle$. 
As we will see in Section \ref{application}, Definition \ref{duplodef}
 is also supported by Proposition  
\eqref{correl-fn}.  
\begin{remark}
Note that due to \eqref{wprop}, in 
Definition \ref{duplodef}, and in \eqref{Z2n_pt_eps1q1} in particular, it is assumed that 
$\F (v_1, x_1;  \ldots; v_k, x_k)$ and $\F(v'_1, y_1; \ldots; v'_n, y_n)$ 
 are composable with the $V$-module $W$ vertex operators 
$Y_W(u, -\zeta_1)$ and $Y_W(\overline{u}, -\zeta_2)$ correspondingly 
(see Section \ref{application} for the definition 
of composability).   
The product \eqref{Z2n_pt_eps1q1} is actually defined by sum of 
products of matrix elements of ordinary $V$-module $W$ vertex operators 
acting on $\W_{z_1, \ldots, z_n}$ elements.
 In what follows 
 we will see that, since $u \in V$ and $\overline{u} \in V'$ 
are connected  by \eqref{dubay}, $\zeta_1$ and $\zeta_2$ appear in a relation to each other. 
The form 
of the product defined above is natural in terms 
of the theory of characters for vertex operator algebras in conformal field theory 
\cite{TUY, FMS, Zhu, MT1, MT2, MT3, MT4, TZ, TZ2} and many others.    
\end{remark}
\subsection{Existence of the corresponding rational form for the $\epsilon$-product}  
In order to describe the $\epsilon$  
product of elements of two spaces $\W_{x_1, \ldots, x_k}$ 
and $\W_{y_1, \ldots, y_n}$ of rational $\W$-valued forms,    
we use a geometric interpretation \cite{H2, Y}.
Recall that a $\W_{z_1, \ldots, z_n}$-space is defined by means 
of matrix elements of the form \eqref{deff}.  
For a vertex algebra $V$, this corresponds \cite{FHL}   
to a matrix element of a number of $V$-vertex operators 
with formal parameters identified with local coordinates on a Riemann sphere. 
Geometrically, each space $\W_{z_1, \ldots, z_n}$ can be also associated to a Riemann sphere
 with a few marked points 
with local coordinates 
 vanishing at these points 
\cite{H2}. 
An extra point 
can be associated to a center of an annulus used in order 
to sew the sphere with another sphere.   
The product \eqref{Z2n_pt_eps1q1} has then a geometric interpretation.    
The resulting model space would also be associated 
to a Riemann sphere formed as a result of sewing procedure.  
In Appendix \ref{sphere} we describe explicitly
 the geometric procedure of sewing of two spheres \cite{Y}.   

Let us identify (as in \cite{H2, Y, Zhu, TUY, FMS, BZF}) 
two sets $(x_1, \ldots, x_k)$ and $(y_1, \ldots, y_n)$ of 
complex formal parameters,
with local
coordinates of two sets of points on the first and the second Riemann spheres correspondingly.   
Identify complex parameters $\zeta_1$, $\zeta_2$ of \eqref{Z2n_pt_eps1q1} with coordinates 
\eqref{disk} of 
the annuluses \eqref{zhopki}.  
After identification of annuluses $\mathcal A_a$ and $\mathcal A_{\overline{a}}$, 
$r$ coinciding coordinates may occur. This takes into account case of coinciding formal 
parameters. 

As we will see in the next Subsection, the product is defined 
by a sum of products of matrix elements  \cite{FHL} associated to each of two spheres.  
Such sum is supposed to describe a $\W$-valued rational differential form defined 
on the sphere formed 
as a result of geometric sewing \cite{Y} of two initial spheres. 
Since two initial spaces $\W_{x_1, \ldots, x_k}$ and $\W_{y_1, \ldots, y_n}$ 
are defined through rational-valued forms 
with the conditions of absolute convergence 
expressed by 
 matrix elements of the form \eqref{deff},  
it is then showed (Proposition \ref{derga}), that the resulting product defines a 
$\W_{x_1, \ldots, x_k; y_1, \ldots, y_n}$-valued rational  
form by means of a convergent matrix element on the resulting sphere. 
In what follows we prove the existence of such rational form.  
The complex sewing parameter,  parameterizing the module space of the sewing spheres,  
parametrizes also the product of the corresponding $\W$-spaces.   

In this Subsection and the next Section we formulate 
the results of this paper for the $\epsilon$-product of $\W$-spaces. 
\begin{proposition}
\label{derga}
The product \eqref{Z2n_pt_eps1q1} of elements of the spaces $\W_{x_1, \ldots, x_k}$ 
and $\W_{y_1, \ldots, y_n}$    
 corresponds to a rational form  
with only possible poles at $x_i=x_j$, $y_{i'}=y_{j'}$, and    
 $x_i=y_{j'}$, $1 \le i, i' \le k$, $1 \le j, j' \le n$.  
\end{proposition}
\begin{proof}
In order to prove this proposition we use the geometric interpretation of the product 
\eqref{Z2n_pt_eps1q1} in terms of the Riemann spheres with marked points  
(see Appendix \ref{sphere}).  
We consider two sets of vertex algebra elements 
$(v_1, \ldots, v_{k})$ and $(v'_1, \ldots, v'_k)$,  
and two sets of  
formal complex parameters 
$(x_1, \ldots, x_k)$, $(y_{1}, \ldots, y_n)$. 
Formal parameters are identified with local coordinates of 
$k$ points on the Riemann sphere $\widehat{\Sigma}^{(0)}_1$, and 
$n$ points on $\widehat{\Sigma}^{(0)}_2$, 
 with excised annuluses ${\mathcal A}_a$ 
(see definitions and notations in Appendix \ref{sphere}).   
 The existence of an appropriate $\W_{x_1, \ldots, x_k; y_1, \ldots, y_n}$-valued  
rational form is clear from Definition \ref{duplodef} 
corresponding to the absolute convergent rational form  
$\mathcal R(x_1, \ldots, x_k; y_1, \ldots, y_n; \epsilon)$ 
defining the $\epsilon$-product of elements of  
the spaces $\W_{x_1, \ldots, x_k}$ and $\W_{y_1, \ldots, y_n}$.    
For all choices of elements of the spaces 
$\W_{x_1, \ldots, x_k}$ and $\W_{y_1, \ldots, y_n}$   
there exists an element  
$\F(v_1, x_1; \ldots; v_k, x_k; 
 v'_1, y_1; \ldots; v'_n, y_n $; $ \epsilon)  
\in \W_{x_1, \ldots, x_k; y_1, \ldots, y_n}$  
such that the product \eqref{Z2n_pt_eps1q1} is 
$R(x_1, \ldots, x_k; y_1, \ldots, y_n; \epsilon)$ $=$ $\langle w'$,   
$\F$ $(v_1, x_1$; $\ldots$; $v_k$, $x_k$; $v'_1$, $y_1$; 
$\ldots$; $v'_n$, $y_n$; $\epsilon)\rangle$.   
\end{proof}  
As we see, the $\epsilon$-product is parametrized by a non-zero complex parameter $\epsilon$, 
and a collection of points on auxiliary  
spheres with 
formal parameters 
vanishing at these points. 
We then have 
\begin{definition}
Let $W$ be a qusi-conformal module for a grading restricted vertex algebra $V$. 
For fixed sets $(v_1, \ldots, v_k)$, $(v'_1, \ldots, v'_n) \in V$,  
$(x_1, \ldots, x_k)\in \C$, $(y_1, \ldots, y_n$) $\in \C$,   
 we call the set of all $\W_{x_1, \ldots, x_k; y_1, \ldots, y_n}$-valued rational forms  
$\F(v_1, x_1; \ldots;  v_k, x_k $ ; $ v'_1, y_1; \ldots;  v'_n,  y_n; \epsilon)$ 
defined by \eqref{Z2n_pt_eps1q1}  
with the parameter $\epsilon$ exhausting all possible values,    
 the complete product of the spaces $\W_{x_1, \ldots, x_k}$ and  $\W_{y_1, \ldots, y_n}$.  
\end{definition}
\section{Properties of the $\W$-product}
\label{properties}
In this Section we study properties of 
the product 
$\F(v_1, x_1$; $ \ldots $; $  v_k, x_k $; $v'_1, y_1; \ldots $; $  v'_n,  y_n$; $\epsilon )$ 
of \eqref{Z2n_pt_eps1q1}.  
Since we assume that $(x_1, \ldots, x_k; y_1, \ldots, y_n)\in F_{k+n}\C$, i.e., 
coincidences of $x_i$ and $y_j$ are excluded by the definition of $F_{k+n}\C$.  We have 
\begin{definition}
We define the action of $\partial_l=\partial_{z_l}={\partial}/{\partial_{z_l}}$, 
$1\le l \le k+n$,   
the differentiation of
$\F(v_1,  x_1;  \ldots;  v_k, x_k$; $v'_1, y_1; \ldots;  v'_n,  y_n; \epsilon)$ 
with respect to the $l$-th entry of 
 $(x_1,  \ldots,   x_k;   y_1, \ldots,   y_n)$  
 as follows 
\begin{eqnarray}
\label{Z2n_pt_eps1qdef}
 & &
\langle w', \partial_l  
\F( v_1,  x_1;  \ldots;  v_k, x_k;   v'_1, y_1; \ldots;  v'_n,  y_n; \epsilon) \rangle  
\nn
 & &  = \sum_{m\in \mathbb{Z} } \epsilon^m \sum_{u\in V_m}     
\langle w',   \partial^{\delta_{l,i}}_{x_i}  
Y^W_{WV}\left( 
 \F ( v_1,  x_1;  \ldots; v_k, x_k), \zeta_1 \right) \; u \rangle  
\nn
& &
\qquad   \qquad  
\langle w',   \partial^{\delta_{l,j}}_{y_j} 
Y^W_{WV}\left(  
\F
(v'_1,  y_1; \ldots;  v'_n,  y_n),  \zeta_2 \right) \overline{u} \rangle.      
\end{eqnarray}
\end{definition}
\begin{proposition}
The product \eqref{Z2n_pt_eps1q1} satisfies the 
 $L_V(-1)$-derivative \eqref{lder1} and $L_V(0)$-conjugation \eqref{loconj} properties. 
\end{proposition}
\begin{proof} 
 By using \eqref{lder1} for $\F(v_1,  x_1;  \ldots; v_k, x_k)$
and $\F(v'_1,  y_1; \ldots;  v'_n,  y_n)$,  
 we consider 
\begin{eqnarray}
\label{Z2n_pt_eps1q00000}
 & &
\langle w',\partial_l 
 \F( v_1,  x_1;  \ldots;  v_k, x_k;   v'_1, y_1; \ldots;  v'_n,  y_n; \epsilon)  \rangle 
\nn
&& 
\nn
 & &  = \sum_{m\in \mathbb{Z}} \epsilon^m \sum_{u\in V_m}    
\langle w', \partial^{\delta_{l,i}}_{x_i} 
Y^W_{WV}\left( 
 \F(v_1,  x_1;  \ldots; v_k, x_k), \zeta_1 \right) \; u \rangle   
\nn
& &
\qquad   \qquad 
\langle w', \partial^{\delta_{l,j}}_{y_j} 
Y^W_{WV}\left(\F(v'_1,  y_1; \ldots;  v'_n,  y_n),  \zeta_2 \right) \overline{u} \rangle     
\nn
 & &  = \sum_{m \in \mathbb{Z} } \epsilon^m \sum_{u\in V_m}    
\langle w', \partial^{\delta_{l,i}}_{x_i} 
Y_W\left( u, - \zeta_1 \right) 
  \F ( v_1,  x_1;  \ldots; v_k, x_k)) \; u \rangle    
\nn
& &
\qquad   \qquad 
\langle w', \partial^{\delta_{l,j}}_{y_j}  
Y_W\left(\overline{u}, - \zeta_2  \right) 
\F(v'_1,  y_1; \ldots;  v'_n,  y_n)  \rangle     
\nn
 & &  = \sum_{m\in \mathbb{Z}}  
\epsilon^m \sum_{u\in V_m}   
\langle w', 
Y^W_{WV}\left( 
\partial^{\delta_{l,i}}_{x_i}  
\F ( v_1,  x_1;  \ldots; v_k, x_k), \zeta_1 \right) \; u \rangle  
\nn
& &
\qquad   \qquad 
\langle w', 
Y^W_{WV}\left( \partial^{\delta_{l,j}}_{y_j} 
\F (v'_1,  y_1; \ldots;  v_n,  y_n),  \zeta_2 \right) \overline{u} \rangle    
\nn
& &  
=  \sum_{m\in \mathbb{Z}} \epsilon^m \sum_{u\in V_m}    
\langle w',
 Y^W_{WV}\left(  
\F (v_1, x_1; \ldots; \left(L_V{(-1)}\right)^{\delta_{l,i}} v_i, x_i; 
\ldots; v_k, x_k), \zeta_1 \right) \; u \rangle  
\nn
& &
\qquad   \qquad 
\langle w', Y^W_{WV}\left( 
\F
(v'_1, y_1; \ldots; \left(L_V{(-1)}\right)^{\delta_{l,j}} v'_j, y_j;   
\ldots; v'_n, y_n),  \zeta_2 \right) \overline{u} \rangle    
\nn
& &  
=    
\langle w',
\F(v_1,  x_1;  \ldots;  \left(L_V{(-1)}\right)_l;  \ldots;  v'_n,  y_n; 
\epsilon) \rangle,   
\end{eqnarray}
where $\left(L_V{(-1)}\right)_l$ acts on the $l$-th entry of  
$(v_1,  \ldots;  v_k;   v'_1,  \ldots,   v'_n)$. 
Summing over $l$ we obtain
\begin{eqnarray}
\label{lder2f0}
&&
\sum\limits_{l=1}^{k+n} \partial_l 
\F (v_1, x_1; \ldots;  
v_k, x_k;
v'_1, y_1; \ldots;  
 v'_n, y_n; \epsilon) \rangle  
\nn
&&= 
\sum\limits_{l=1}^{k+n} \langle w', 
\F (v'_1, x_1; \ldots; \left(L_V{(-1)}\right);    \ldots;  v'_n, y_n; \epsilon) \rangle  
\nn
&&
=\langle w',
 L_W{(-1)}.\F(v_1, x_1;  \ldots ; v_k, x_k;   v'_1,  y_1; \ldots; v'_n,  y_n; 
\epsilon) \rangle.    
\end{eqnarray}
Due to \eqref{loconj}, \eqref{locomm}, \eqref{dubay}, \eqref{condip}, and \eqref{aprop},  
we have 
\begin{eqnarray*}
&& \langle w', 
 \F ( z^{L_V{(0)} } v_1, z \; x_1;  \ldots;  z^{L_V{(0)}} v_k,  z\; x_k; 
z^{L_V{(0)}} v'_1, z \; y_1;  \ldots;  z^{L_V{(0)}} v'_n,  z\; y_n;  
 \epsilon) \rangle  
\end{eqnarray*}
\begin{eqnarray*} 
 & &  = \sum_{m \in \mathbb{Z}} \epsilon^m \sum_{u\in V_m}    
\langle w', Y^W_{WV}\left(  
 \F (z^{L_V{(0)}} v_1, z\; x_1;  \ldots; z^{L_V{(0)}} v_k, z\; x_k), \zeta_1 \right) \;  
u \rangle   
\nn
& &
\qquad   \qquad 
\langle w', Y^W_{WV}\left( 
\F (z^{L_V{(0)}} v'_1, z \; y_1; \ldots; z^{L_V{(0)}} v'_n, z\; y_n),  \zeta_2 \right)  
\end{eqnarray*}
\begin{eqnarray*}
& &  
=  \sum_{m \in \mathbb{Z} } \epsilon^m \sum_{u\in V_m}    
\langle w', Y^W_{WV}\left( z^{L_V{(0)}}  
\F (v_1, x_1;  \ldots; v_k, x_k), \zeta_1 \right) \; u \rangle  
\nn
& &
\qquad   \qquad 
\langle w', Y^W_{WV}\left( z^{L_V{(0)}} 
\F (v'_1, y_1; \ldots; v'_n, y_n),  \zeta_2 \right) \overline{u} \rangle    
\end{eqnarray*}
\begin{eqnarray*}
& &  
=  \sum_{m \in \mathbb{Z}} \epsilon^m \sum_{u\in V_m}    
\langle w', e^{\zeta_1 L_W{(-1)}}  Y_W \left(  u, -\zeta_1 \right)  
z^{L_V{(0)}}  
\F (v_1, x_1;  \ldots; v_k, x_k)  \rangle 
\nn
& &
\qquad   \qquad 
\langle w', e^{\zeta_2 L_W{(-1)}} \; Y_W\left( \overline{u}, -\zeta_2  \right) z^{L_V{(0)}}  
\F (v'_1, y_1; \ldots; v'_n, y_n)   \rangle     
\end{eqnarray*}
\begin{eqnarray*}
& &  
=  \sum_{ m\in \mathbb{Z} } \epsilon^m \sum_{u\in V_m}    
\langle w', e^{\zeta_1 L_W{(-1)}}  z^{L_V{(0)}} 
Y_W \left(  z^{-L_V{(0)}} u, -z \; \zeta_1 \right)  
\F (v_1, x_1;  \ldots; v_k, x_k)  \rangle 
\nn
& &
\qquad   \qquad 
\langle w', e^{\zeta_2 L_W{(-1)}}\; z^{L_W{(0)}}\;   
Y_W \left( z^{-L_V{(0)}} \overline{u}, -z \; \zeta_2  \right)   
\F (v'_1, y_1; \ldots; v'_n, y_n)   \rangle    
\nn
& &  
=  \sum_{m\in \mathbb{Z} } \epsilon^m \sum_{u\in V_m}    
\langle w', e^{\zeta_1 L_W{(-1)}}  z^{L_W{(0)}} z^{-{\rm wt} u} \; Y_W \left(u, -z \;  
\zeta_1 \right) 
\F (v_1, x_1;  \ldots; v_k, x_k)  \rangle   
\nn
& &
\qquad   \qquad 
\langle w', e^{\zeta_2 L_W{(-1)}}\; z^{L_W(0)}\;  z^{-{\rm wt}\overline{u} } 
\; Y_W\left(  \overline{u}, -z \; \zeta_2  \right)   
\F (v'_1, y_1; \ldots; v'_n, y_n)   \rangle   
\end{eqnarray*}
\begin{eqnarray*}
& &  
=  \sum_{m \in \mathbb{Z} } \epsilon^m \sum_{u\in V_m}    
\langle w', z^{L_W(0)} e^{\zeta_1 L_W{(-1)}} Y_W \left(  u, -z \zeta_1 \right)  
\F (v_1, x_1;  \ldots; v_k, x_k)  \rangle   
\nn
& &
\qquad   \qquad  
\langle w', z^{L_W(0)} e^{\zeta_2 L_W{(-1)}}  Y_W\left( \overline{u}, -z \zeta_2  \right)  
\F (v'_1, y_1; \ldots; v'_n, y_n),   \rangle    
\end{eqnarray*}
\begin{eqnarray*}
& &  
=  \sum_{m\in \mathbb{Z} } \epsilon^m \sum_{u\in V_m}    
\langle w', z^{L_W(0)} \; Y^W_{WV}\left( 
\F (v_1, x_1;  \ldots; v_k, x_k), z \zeta_1 \right) \; u \rangle  
\notag 
\nn
& &
\qquad   \qquad 
\langle w', z^{L_W(0)} \;  Y^W_{WV}\left( 
\F (v'_1, y_1; \ldots; v'_n, y_n ),  z \zeta_2 \right) \overline{u} \rangle   
\end{eqnarray*}
\begin{eqnarray*}
& &  
=  \sum_{m \in \mathbb{Z} } \epsilon^m  \sum_{u\in V_m}    
\langle w', z^{L_W(0)} \; Y^W_{WV}\left(  
\F (v_1, x_1;  \ldots; v_k, x_k),  \zeta'_1 \right) \; u \rangle   
\nn
& &
\qquad   \qquad 
\langle w', z^{L_W(0)} \;  Y^W_{WV}\left( 
\F (v'_1, y_1; \ldots; v'_n, y_n),  \zeta'_2 \right) \overline{u} \rangle    
\end{eqnarray*}
\begin{eqnarray*}
&&
=\langle w', \left( z^{L_W(0)} \right).  
\F (v_1, x_1;  \ldots;  v_k, x_k; v'_1, y_1; \ldots; v'_n, y_n; \epsilon) \rangle.    
\end{eqnarray*}
With \eqref{pinch}, 
we obtain \eqref{loconj} for \eqref{Z2n_pt_eps1q1}.   
\end{proof}
\begin{remark}
As we see in the last expressions, 
  the $L_V(0)$-conjugation property \eqref{loconj} for the product \eqref{Z2n_pt_eps1q1} 
includes the action 
of $z^{L_V(0)}$-operator on complex parameters $\zeta_a$, $a=1$, $2$.   
\end{remark}
We also have 
\begin{proposition}
\label{pupa}
For primary elements $v_i$, $v'_j \in V$, $1 \le i \le k$, 
$1 \le j \le n$,  
of a quasi-conformal grading-restricted vertex algebra $V$ and 
its grading-restricted generalized module $W$,   
the product \eqref{Z2n_pt_eps1q1} is canonical with respect to the action of the group  
${\rm Aut}_{x_1, \ldots, x_{k}; y_1, \ldots, y_{n}}\Oo^{(k+n)}$   
of 
$k+n$-dimensional 
changes 
\begin{eqnarray}
\label{zwrho}
&&(x_1, \ldots, x_k; y_1, \ldots, y_n) \mapsto (x'_1, \ldots, x'_k; y'_1, \ldots, y'_n) 
\nn
&&
\qquad =(\rho_1(x_1, \ldots, x_k; y_1, \ldots, y_n), \ldots, 
\rho_{k+n}(x_1, \ldots, x_k; y_1, \ldots, y_n)), 
\end{eqnarray} 
of formal parameters. 
\end{proposition}
\begin{proof}
Note that due to Proposition \ref{ndimwinv} 
\begin{eqnarray*}
\F (v_1, x'_1;  \ldots; v_k, x'_k) &=&
\F (v_1, x_1;  \ldots; v_k, x_k), 
\nn
\F (v_1, y'_1;  \ldots; v_n, y'_n) &=& 
\F (v_1, y_1;  \ldots; v_n, y_n).   
\end{eqnarray*}
Thus,  
\begin{eqnarray*}
&& \langle w', \F(v_1, x'_1; \ldots,;  
v_k, x'_k; v'_1, y'_1; \ldots; v'_n, y'_n; \epsilon) \rangle 
\nn
 & &  =  \sum_{l\in \mathbb{Z}} \epsilon^l \sum_{u\in V_l}    
\langle w', Y^W_{WV}\left(  
\F (v_1, x'_1;  \ldots; v_k, x'_k), \zeta_1 \right) \; u \rangle   
\nn
& &
\qquad \qquad  
\langle w', Y^W_{WV}\left( 
\F (v'_1, y'_1; \ldots; v'_n, y'_n),  \zeta_2 \right) \overline{u} \rangle   
\nn
& &  
=  \sum_{l\in \mathbb{Z} } \epsilon^l \sum_{u\in V_l}   
 \langle w', Y^W_{WV}\left(  
\F (v_1, x_1;  \ldots; v_k, x_k), \zeta_1 \right) \; u \rangle 
\nn
& &
\qquad   \qquad  
\langle w', Y^W_{WV}\left( 
\F (v'_1, y_1; \ldots; v'_n, y_n),  \zeta_2 \right) \overline{u} \rangle    
\nn
&&
= \langle w', \F(v_1, x_1; \ldots,;  v_k, x_k; v'_1, y_1; \ldots; v'_n, y_n; \epsilon) \rangle. 
\end{eqnarray*}
Thus, the product \eqref{Z2n_pt_eps1q1} is invariant under \eqref{zwrho}. 
\end{proof}
In the geometric interpretation in terms of auxiliary spaces,  
the definition \eqref{Z2n_pt_eps1q1} depends on the choice of 
insertion points $p_i$, $1 \le i \le k$,   
with local coordinates $x_i$ 
on $\widehat{\Sigma}^{(0)}_1$,   
and $p'_i$, $1 \le j \le k$, with local coordinates $y_j$ on  
 $\widehat{\Sigma}_2^{(0)}$.
 Suppose we change the the distribution of points among two Riemann spheres.  
We formulate the following   
\begin{lemma}
\label{functionformpropcor} 
For a fixed set $(\widetilde{v}_1 , \ldots, \widetilde{v}_n) \in V$,  
of vertex algebra elements,  
 the $\epsilon$-product
$\F(\widetilde{v}_1, z_1;  \ldots;  \widetilde{v}_n, z_n; \epsilon) 
\in \W_{z_1, \ldots,  z_n}$,  
  \begin{equation}
\label{invar}
\cdot_\epsilon : \W_{z_1, \ldots, z_k} \times  \W_{z_{k+1}, \ldots, z_n} \rightarrow   
 \W_{z_1, \ldots, z_n}, 
\end{equation}
 remains the same  
 for elements $\F(\widetilde{v}_1, z_1;  \ldots;  \widetilde{v}_k, z_k) \in  
  \W_{z_1, \ldots, z_k}$ and 
$\F(\widetilde{v}_{k+1}, z_{k+1} $; $  \ldots $ ; 
$ \widetilde{v}_n, z_n) \in \W_{z_{k+1},  \ldots, z_n}$,    
for any $0 \le k \le n$.  
\end{lemma}
\begin{remark}
This Lemma is important for the formulation of cohomology 
invariants associated to grading-restricted vertex algebras  
on smooth manifolds.   
In case $k=0$, we obtain from \eqref{invar}, 
\begin{equation}
\label{invar}
\cdot_\epsilon: \W \times \W_{z_1, \ldots, z_n} \rightarrow  
 \W_{z_1, \ldots, z_n}.   
\end{equation}
\end{remark}
\begin{proof}
Let $\widetilde{v}_i \in V$, $1 \le i \le k$,   
$\widetilde{v}_j \in V$, $1 \le j \le k$,    
and $z_i$, $z_j$ are corresponding formal parameters.  
We show that  
the $\epsilon$-product of 
$\F(\widetilde{v}_1, z_1;  \ldots;  \widetilde{v}_k, z_k)$ and  
$\F(\widetilde{v}_{k+1}, z_{k+1};  \ldots; \widetilde{v}_n, z_n)$, i.e.,
 the $\W_{z_1, \ldots, z_{k+n}}$-valued  
differential form 
\begin{eqnarray}
& & {\mathcal F} 
 ((\widetilde{v}_1, z_1;  \ldots;  \widetilde{v}_k, z_k);  
(\widetilde{v}_{k+1}, z_{k+1};  \ldots; \widetilde{v}_n, z_n);   
\zeta_1, \zeta_2; \epsilon )     
\end{eqnarray}
is independent of the choice of $0 \le k \le n$.   
 Consider  
\begin{eqnarray}
\label{ishodnoe}
&& 
\langle w',  
\F (\widetilde{v}_1, z_1; \ldots ;  \widetilde{v}_k, z_k; \widetilde{v}_{k+1}, z_{k+1}; 
 \ldots ; \widetilde{v}_n, z_n; \zeta_1, \zeta_2; \epsilon)  \rangle    
\nn
&&\qquad  = 
 \sum_{l\in \mathbb{Z} } \epsilon^l \sum_{u \in V_l}    
\langle w', Y^W_{WV}\left(  
\F (\widetilde{v}_1, z_1;  \ldots; \widetilde{v}_k, z_k), \zeta_1 \right) \; u \rangle  
\nn
& &
\qquad   \qquad  \qquad 
\langle w', Y^W_{WV}\left( 
\F (\widetilde{v}_{k+1}, z_{k+1}; \ldots; 
\widetilde{v}_n, z_n),  \zeta_2 \right) \overline{u} \rangle.      
\end{eqnarray}
On the other hand, for $0 \le m \le k$,  consider
\begin{eqnarray*}
&& 
 \sum_{l\in \mathbb{Z} } \epsilon^l \sum_{u \in V_l}    
\langle w', Y^W_{WV}\left( 
\F (\widetilde{v}_1, z_1;  \ldots; \widetilde{v}_m, z_m), \zeta_1 \right) \; u \rangle 
\nn
& &
\qquad   \qquad  \qquad 
\langle w', Y^W_{WV}\left( 
\F( \widetilde{v}_{m+1}, z'_{m+1}; \ldots; \widetilde{v}_k, z'_k;   
\widetilde{v}_{k+1}, z_1; \ldots; \widetilde{v}_n, z_n),  \zeta_2 \right) \overline{u} \rangle 
\nn
&&
\qquad =\langle w',  
\F (\widetilde{v}_1, z_1; \ldots ;  \widetilde{v}_m, z_m;
\widetilde{v}_{m+1}, z'_{m+1}; \ldots; \widetilde{v}_k, z'_k;
 \widetilde{v}_{k+1}, z_{k+1};
 \ldots ; \widetilde{v}_n, z_n)  \rangle. 
\end{eqnarray*}
The last is the $\epsilon$-product \eqref{Z2n_pt_eps1q1} of 
$\F (\widetilde{v}_1, z_1;  \ldots; \widetilde{v}_m, z_m)  \in 
\W_{z_1,  \ldots, z_m}$ and  $\F 
(\widetilde{v}_{m+1}, z'_{m+1} $; $ \ldots $; $ \widetilde{v}_k, z'_k $; $    
\widetilde{v}_{k+1}, z_1$; $ \ldots $;  $\widetilde{v}_n, z_n) 
\in \W_{z'_{m+1}, \ldots, z'_k;   
 z_1, \ldots, z_n}$.  
Let us apply the invariance with respect to a subgroup of
 ${\rm Aut}_{z_1, \ldots, z_{k+n}}\; \Oo^{(n)}$,   
with $(z_1, \ldots, z_m)$ and $(z_{k+1}, \ldots, z_n)$ remaining unchanged.  
Then we obtain the same product \eqref{ishodnoe}.  
\end{proof}
Next, we formulate 
\begin{definition}
\label{sprod}
We define the
action of an element $\sigma \in S_{k+n}$ on the product of 
$\F (v_1, x_1;  \ldots; v_k, x_k) \in \W_{x_1, \ldots, x_k}$, and  
$\F (v'_1, y_1; \ldots; v'_n, y_n) \in \W_{y_1, \ldots, y_n}$, as 
\begin{eqnarray}
\label{Z2n_pt_epsss}
&& \langle w',  \sigma(\F) (v_1, x_1;  \ldots; v_k, x_k; v'_1, y_1; \ldots; v'_n,  y_n;  
\epsilon) \rangle 
\nn
&&
\qquad =\langle w',  \F (v_{\sigma(1)}, x_{\sigma(1)};  \ldots; v_{\sigma(k)}, x_{\sigma(k)}; 
v'_{\sigma(1)}, y_{\sigma(1)}; \ldots; v'_{\sigma(n)}, y_{\sigma(n)};  
\epsilon) \rangle 
\nn 
& & \qquad =  
\sum_{u\in V }  
 \langle w', Y^W_{WV}\left(  
\F (v_{\sigma(1)}, x_{\sigma(1)};  
\ldots; v_{\sigma(k)}, x_{\sigma(k)}), \zeta_1\right)\; u \rangle   
\nn
& &
\qquad   \qquad  \qquad  
 \langle w', Y^W_{WV}\left(  
\F (v'_{\sigma(1)}, y_{\sigma(1)}; \ldots; v'_{\sigma(n)}, y_{\sigma(n)}) , \zeta_2\right) \;  
\overline{u} \rangle.     
\end{eqnarray}
\end{definition}
\section{Spaces $C^n_m(V, \W)$ of a complex} 
\label{application}
 The spaces $C^n_m(V, \W)$, $n\ge 0$, $m \ge 0$, and differentials $\delta^n_m$,  
for the chain-cochain 
complex $(C^n_m(V, \W), \delta^n_m)$ were introduced in \cite{Huang, FQ}). 
In this Section we recal the definition and properties of $C^n_m(V, \W)$.   
\subsection{$E$-elements}
 For $w\in W$, the $\overline{W}$-valued function  
 given by 
$$
E^{(n)}_W(v_1, z_1; \ldots; v_n, z_n; w) 
= E(\omega_W(v_1, z_1) \ldots \omega_W(v_n, z_n)w),    
$$
where an element $E(\phi)$ is a $\overline{W}$-valued rational function, 
 \begin{equation}
\label{poper}
\omega_W\left(
v_i,  
z_i\right)
= Y_W\left(  
dz_i^{\wt(v_i)} \otimes v_i,  
z_i)\right),   
\end{equation}
 $\phi\in \overline{W}$ is given by 
(see notations for $\omega_W(.,.)$ in Section \ref{spaces}) 
\[
E(\phi) 
=R(\langle w', \phi \rangle).  
\]
One defines  
\[
E^{W; (n)}_{WV}(w; v_1, z_1 ; \ldots; v_n, z_n)   
=E^{(n)}_W(v_1, z_1; \ldots;  v_n, z_n; w),
\]
where 
$E^{W; (n)}_{WV}(w; v_1, z_1 ; \ldots; v_n, z_n)$ is   
an element of $\overline{W}_{z_1, \dots, z_n}$.
For $(z_1, \dots, z_n, \zeta)\in 
F_{n+1}\C$, $(v_1, \dots, v_n)\in V$, and $w\in W$,  
set 
\begin{eqnarray*}
  E^{(n, 1)}_W(v_1, z_1;  \ldots; v_n, z_n; w, \zeta)  
 =E\left(Y_W(v_1, z_1)\ldots Y_W(v_n, z_n) \; Y_{WV}^W(w, \zeta) \one_V \right). 
\end{eqnarray*}
One defines
\[
 \F \circ \left(E^{(l_1)}_{V;\;\one}\otimes \ldots \otimes E^{(l_n)}_{V;\;\one}\right): 
V^{\otimes m+n}\to \overline{W}_{z_1,  \dots, z_{m+n}}, 
\] 
by
\begin{eqnarray*}
&&(\F\circ (E^{(l_1)}_{V;\;\one}\otimes \ldots \otimes  
E^{(l_{n})}_{V;\;\one}))(v_1 \otimes \ldots \otimes v_{m+n-1}) 
\nn
&&=E(\F(E^{(l_1)}_{V; \one}(v_1 \otimes \ldots \otimes v_{l_1})\otimes \ldots
\nn 
&&\quad\quad\quad\quad\quad \otimes 
E^{(l_n)}_{V; \one} 
(v_{l_1+\ldots +l_{n-1}+1}\otimes \ldots  
\otimes v_{l_1+\ldots +l_{n-1}+l_n}))),   
\end{eqnarray*}
and 
\[
E^{(m)}_W \circ_0 \F: V^{\otimes m+n}\to 
\overline{W}_{z_1, \dots, z_{m+n-1}}, 
\]
 is given by 
\begin{eqnarray*}
&& 
(E^{(m)}_W\circ_0 \F)(v_1\otimes \ldots \otimes v_{m+n}) 
\nn 
&&
=E(E^{(m)}_W(v_1 \otimes \ldots\otimes v_m;
\F(v_{m+1}\otimes \ldots\otimes v_{m+n}))).
\end{eqnarray*}
Finally,  
\[
E^{W; (m)}_{WV}\circ_{m+1} \F: V^{\otimes m+n}\to  
\overline{W}_{z_1, \dots, z_{m+n-1}}, 
\]
 is defined by 
\begin{eqnarray*}
(E^{W; (m)}_{WV}\circ_{m+1 }\F)(v_1 \otimes \ldots \otimes v_{m+n}) 
 =E(E^{W; (m)}_{WV}(\F(v_1 \otimes \ldots\otimes v_n)
; v_{n+1}\otimes \ldots\otimes v_{n+m})). 
\end{eqnarray*}
In the case that $l_1=\ldots=l_{i-1}=l_{i+1}=1$ and $l_i=m-n-1$, for some $1 \le i \le n$, 
we will use $\F\circ_i E^{(l_i)}_{V;\;\one}$ to 
denote $\F\circ (E^{(l_1)}_{V;\;\one}\otimes \ldots  
\otimes E^{(l_n)}_{V;\;\one})$. 
\subsection{Maps composable with vertex operators}
\label{composable}
Let us recall the definition of maps composable with a number of vertex operators \cite{Huang}. 
\begin{definition}
\label{composabilitydef}
For a $V$-module 
\[
W=\coprod_{n\in \C} W_{(n)}, 
\]
 and $m\in \C$, 
let
\[
P_m: \overline{W}\to W_{(m)},  
\]
 be the projection from 
$\overline{W}$ to $W_{(m)}$.  
Let 
\[
\F: V^{\otimes n} \to \W_{z_1, \dots, z_n},  
\]
 be a map. For $m\in \N$,  
$\F$ is called 
composable with $m$ vertex operators if  
the following conditions are satisfied:

1)  Let $l_1, \dots, l_n \in \Z_+$ such that $l_1+\ldots +l_n=m+n$, 
$v_1, \dots, v_{m+n}\in V$, and $w'\in W'$. Set  
 \begin{eqnarray}
\label{psii}
\Psi_i 
&
=
&
E^{(l_i)}_V(v_{k_1}, z_{k_1}- \zeta_i; 
 \ldots; 
v_{k_i}, z_{k_i}- \zeta_i  
 ; \one_V),     
\end{eqnarray}
where
\begin{eqnarray}
\label{ki}
 {k_1}={l_1+\ldots +l_{i-1}+1}, \quad  \ldots, \quad  {k_i}={l_1+\ldots +l_{i-1}+l_i},  
\end{eqnarray} 
for $i=1$, $\dots$, $n$. Then there exist positive integers $N^n_m(v_i, v_j)$  
depending only on $v_i$ and $v_j$ for $i$, $j=1$, $\dots$, $k$, $i\ne j$ such that the series 
\begin{eqnarray}
\label{Inm0}
\mathcal I^n_m(\F)=
\sum_{r_1, \dots, r_n \in \Z}\langle w',  
\F(P_{r_1}\Psi_1; \zeta_1;  
 \ldots; 
P_{r_n} \Psi_n, \zeta_n) 
\rangle,
\end{eqnarray} 
is absolutely convergent when  
\begin{eqnarray}
\label{granizy1}
|z_{l_1+\ldots +l_{i-1}+p}-\zeta_i| 
+ |z_{l_1+\ldots +l_{j-1}+q}-\zeta_i|< |\zeta_i -\zeta_j|, 
\end{eqnarray} 
for $i$, $j=1$, $\dots$, $k$, $i\ne j$ and for $p=1$,  
$\dots$,  $l_i$ and $q=1$, $\dots$, $l_j$.  
The sum must be analytically extended to a
rational function
in $(z_1, \dots, z_{m+n})$, 
 independent of $(\zeta_1, \dots, \zeta_n)$, 
with the only possible poles at 
$z_i=z_j$, of order less than or equal to 
$N^n_m(v_i, v_j)$, for $i$, $j=1$, $\dots$, $k$,  $i\ne j$.  

2)
 For $v_1, \dots, v_{m+n}\in V$, and $(z_1, \ldots, z_{n+m})\in \C$  there exist 
positive integers $N^n_m(v_i, v_j)$, depending only on $v_i$ and  
$v_j$, for $i$, $j=1$, $\dots$, $k$, $i\ne j$, such that for arbitrary $w'\in W'$, and 
 such that the series  
\begin{eqnarray}
\label{Jnm0}
\mathcal J^n_m(\F)=  
\sum_{q\in \C} \langle w', E^{(m)}_W \Big(v_1, z_1; \ldots; 
v_m, z_m;   
P_q( \F( v_{m+1}, z_{m+1}; \ldots; v_{m+n}, z_{m+n}) \Big)\rangle,  
\end{eqnarray} 
is absolutely convergent when 
\begin{eqnarray}
\label{granizy2}
z_i \ne z_j, \quad i\ne j, \quad 
\nn
|z_i|>|z_k|>0, 
\end{eqnarray}
 for $i=1$, $\dots$, $m$, 
and $k=m+1$, $\dots$, $m+n$, and the sum can be analytically extended to a 
rational function 
in $(z_1, \dots, z_{n+m})$ with the only possible poles at  
$z_i=z_j$, of orders less than or equal to 
$N^n_m(v_i, v_j)$, for $i$, $j=1$, $\dots$, $k$, $i\ne j$.  
\end{definition} 
In \cite{Huang}, we the following useful proposition was proven: 
\begin{proposition}
\label{comp-assoc}
Let $\F: V^{\otimes n}\to \W_{z_1, \dots, z_n}$  
be composable with $m$ vertex operators. Then we have:
\begin{enumerate}
\item For $p\le m$, $\F$ is 
composable with $p$ vertex operators and for
$p$, $q\in \Z_+$ such that $p+q\le m$ and 
$l_1, \dots, l_n \in \Z_+$ such that $l_1+\ldots +l_n=p+n$, 
$\F\circ (E^{(l_1)}_{V;\;\one}\otimes  
\ldots \otimes E^{(l_n)}_{V;\;\one})$ and $E^{(p)}_W \circ_{p+1}\F$  
are composable with $q$ vertex operators.

\item For $p$, $q\in \Z_+$ such that $p+q\le m$, 
$l_1$, $\dots$, $l_n \in \Z_+$ such that $l_1 +\ldots +l_n=p+n$ and
$k_1, \dots, k_{p+n} \in \Z_+$ such that $k_1+\ldots +k_{p+n}=q+p+n$,
we have
\begin{eqnarray*}
&(\F\circ (E^{(l_{1})}_{V;\;\one}\otimes 
\ldots \otimes E^{(l_{n})}_{V;\;\one}))\circ 
(E^{(k_{1})}_{V;\;\one}\otimes 
\ldots \otimes E^{(k_{p+n})}_{V;\;\one})&\nn
&=\F\circ (E^{(k_{1}+\ldots +k_{l_{1}})}_{V;\;\one}\otimes 
\ldots \otimes E^{(k_{l_{1}+\ldots +l_{n-1}+1}+\ldots +k_{p+n})}_{V;\;\one}).& 
\end{eqnarray*}
\item For $p$, $q\in \Z_+$ such that $p+q\le m$ and 
$l_1, \dots, l_n \in \Z_+$ such that $l_1+\ldots +l_n=p+n$,
we have
$$E^{(q)}_W\circ_{q+1} (\F\circ (E^{(l_1)}_{V;\;\one}\otimes 
\ldots \otimes E^{(l_{n})}_{V;\;\one}))
=(E^{(q)}_W\circ_{q+1} \F)\circ (E^{(l_1)}_{V;\;\one}\otimes 
\ldots \otimes E^{(l_n)}_{V;\;\one}).$$ 
\item For $p$, $q\in \Z_{+}$ such that $p+q\le m$, we have 
$$E^{(p)}_W\circ_{p+1} (E^{(q)}_W\circ_{q+1}\F) 
=E^{(p+q)}_W\circ_{p+q+1}\F.$$ 
\end{enumerate}
\end{proposition}
Finally, in \cite{Huang} we find the proof of the following   
\begin{proposition}
\label{correl-fn}
Let now $P_n: W\to W_{(n)}$, 
 for $n\in \C$ be the projection from $W$ to $W_{(n)}$. 
For $k$, $l_1, \dots, l_{n+1} \in \Z_+$ and 
$v_1^{(1)},\dots, v_{l_1}^{(1)},\dots, v_1^{(n+1)}, \dots$, 
$v_{l_{n+1}}^{(n+1)} \in V$, $w\in W$, and $w'\in W'$, the series 
\begin{eqnarray}
\label{va-conv-axiom}
&&\sum_{r_1, \ldots, r_n \in \Z, r_{n+1}\in \C}  
\langle w', E^{(n, 1)}_W \big( P_{r_1}  
(E^{(l_1)}_V(v_1^{(1)}, z_1^{(1)};  
\ldots;  
 v_{l_1}^{(1)}, z_{l_1}^{(1)}; \one_V, z^{(0)}_1) ); \ldots;  
\nn
&& \qquad \qquad  
P_{r_n} (E^{(l_n)}_V(v_1^{(n)}, z_1^{(n)};  
\ldots; 
 v_{l_n}^{(n)}, z_{l_n}^{(n)}; \one_V, z^{(0)}_n))   
\nn
&& \qquad \qquad
 P_{r_{n+1}}(E^{(l_{n+1})}_W(v_1^{(n+1)}, z_1^{(n+1)};  
\ldots; 
 v_{l_{n+1}}^{(n+1)}, z_{l_{n+1}}^{(n+1)}; w, z^{(0)}_{n+1})) \Big) 
\rangle,  
\end{eqnarray}
converges absolutely to  
\begin{eqnarray*} 
&&
 \langle w', E^{(n)}_W (  v_1^{(1)}, z_1^{(1)}+z^{(0)}_1; \ldots;   
v_{l_1}^{(1)}, z_{l_1}^{(1)}+z^{(0)}_1; \ldots;  
\nn
&&
\qquad \qquad 
 v_1^{(n+1)}, z_1^{(n+1)}+z^{(0)}_{n+1};  
v_{l_{n+1}}^{(n+1)},  z_{l_{n+1}}^{(n+1)}+z^{(0)}_{n+1}; w)  
 )\rangle, 
\end{eqnarray*}
when $0<|z_p^{(i)}| + |z_q^{(j)}|< |z^{(0)}_i-z^{(0)}_j|$  
for $i$, $j=1, \dots, n+1$, $i\ne j$, $p=1, \dots,  l_i$, $q=1, \dots, l_j$.
\end{proposition}
\subsection{Definition of $C^n_m(V, \W)$-spaces} 
\label{spaces}
In this Subsection we recall the definition of spaces $C^n_m(V, \W)$   
 for a grading-restricted vertex algebra $V$. 
 First, recall the definition of shuffles \cite{Huang}.   
For $l \in \N$ and $1\le s \le l-1$, let $J_{l; s}$ be the set of elements of 
$S_l$ which preserve the order of the first $s$ numbers and the order of the last 
$l-s$ numbers, that is, 
\[
J_{l; s}=\{\sigma\in S_l\;|\;\sigma(1)<\ldots <\sigma(s),\; 
\sigma(s+1)<\ldots <\sigma(l)\}.
\]
The elements of $J_{l; s}$ are called shuffles, and we use the notation 
\[
J_{l; s}^{-1}=\{\sigma^{-1}\;|\; \sigma\in J_{l; s}\}.
\]
For a set of $n$ elements $\left(v_1, \ldots, v_n\right)$  
of a grading-restricted vertex algebra $V$,   
we consider  maps 
\begin{equation} 
\label{maps}
\F\left(v_1, z_1; \ldots; v_n, z_n \right): V^{\otimes n} \rightarrow \W_{z_1, \ldots, z_n} 
\end{equation}
 (see Section \ref{valued} for the definition of a 
$\W_{z_1, \dots, z_n}$ space).  
Note that similar to considerations of \cite{BZF}, 
 \eqref{bomba}  
 can be treated as 
${\rm Aut}_{z_1, \ldots, z_n}\; \Oo^{(n)}$  
-torsor  
of the product of groups of formal parameter transformations.   
 In what follows, according to definitions of Appendix \ref{valued}, 
 when we write an element $\F$ of the space $\W_{z_1, \ldots, z_n}$, we actually have in mind 
 corresponding matrix element $\langle w', \F\rangle$ that 
  absolutely converges (in a certain domain) to 
a rational form-valued function $R(\langle w', \F \rangle)$. Quite frequently we will write 
$\langle w', \F \rangle$ which would denote a rational $\W$-valued form.    
In notations, we would keep tensor products of vertex algebra elements with 
$\wt$-powers of $z$-differentials  
when it is inevitable only. 

 In the next Section we prove, that 
 for arbitrary $v_i \in V$, $1 \le i \le n$,  
with formal parameters $z_i$ 
an element \eqref{bomba} 
 as well as the vertex operators \eqref{poper} 
 are invariant with respect to the action of the group 
${\rm Aut}_{z_1, \ldots, z_n}\; \Oo^{(n)}$ 
In \eqref{poper} we mean the ordinary vertex operator (as defined in Appendix \ref{grading})
 not affecting the tensor product 
with corresponding differential. 
In \cite{Huang} one finds:
\begin{proposition} 
\label{tudaty}
The subspace of $\hom(V^{\otimes n}, 
\W_{z_1, \dots, z_n})$ consisting of  maps having
the $L_V(-1)$-derivative property, having the $L_V(0)$-conjugation property
or being composable with $m$ vertex operators is invariant under the 
action of $S_n$. \hfill $\qed$
\end{proposition}
 We next have 
\begin{definition}
\label{initialspace}
For an arbitrary set of vertex algebra elements $v_i$, $v_j \in V$,  
and formal complex parameters $z_i$, $z_j$, $1\le i \le n$, 
$1\le j \le m$,  
 $n \ge 0$, $m \ge 0$,    
we denote by $C^n_m(V, \W)$,      
the space of all maps \eqref{maps} 
\begin{equation}
\label{mapy}
 \F(v_1, z_1; \ldots; v_n, z_n): V^{\otimes n } \rightarrow \W_{z_1,
\ldots, z_n},   
\end{equation} 
composable with a $m$ of vertex operators \eqref{poper}  
with vertex algebra elements $v_j$, 
 with formal parameters $z_j$.   
We assume also that \eqref{bomba}   
satisfy $L_V{(-1)}$-derivative \eqref{lder1}, $L_V(0)$-conjugation 
\eqref{loconj} properties, and the symmetry property 
with respect to action of the symmetric group $S_n$: 
\begin{equation}
\label{shushu} 
\sum_{\sigma\in J_{n; s}^{-1}}(-1)^{|\sigma|} 
\F\left(v_{\sigma(1)}, z_{\sigma(1)}; \ldots; v_{\sigma(n)},  z_{\sigma(n)} \right) =0. 
\end{equation}
\end{definition}

In Appendix \ref{proof} we give a proof of 
the following 
\begin{proposition} 
\label{nezc}
For primary vectors of a quasi-conformal grading-restricted vertex algebra $V$,      
 Definition \ref{initialspace} is canonical, i.e.,
invariant with respect to the group of $n$-dimensional 
transformations
\[ 
 (z_1, \ldots, z_n) \mapsto (z'_1, \ldots, z'_n)
= (\rho_1(z_1, \ldots, z_n), \ldots, \rho_n (z_1, \ldots, z_n) ),  
\] 
of formal parameters $z_i$,  $1 \le i \le n$.  
\end{proposition}
In Appendix \ref{proof} we recall the proof of Proposition \ref{nezc}.  
\begin{remark}
 The condition of quasi-conformality is necessary  
in the proof of invariance of elements of the space  
$\W_{z_1, \ldots, z_n}$  
with respect to a vertex algebraic representation (cf. Appendix \ref{grading}) of the group 
${\rm Aut}_{z_1, \ldots, z_n}\; \Oo^{(n)}$.   
In what follows,  
 we will always assume the quasi-conformality of $V$-modules  
 when it concerns the spaces $C^n_m(V, \W)$. 
\end{remark}
\subsection{Coboundary operators} 
\label{coboundary}
In this Subsection we recall 
 the definition of 
 the coboundary operator 
for the spaces $C^n_m(V, \W)$, 
\begin{eqnarray}
\label{hatdelta} 
  \delta^n_m \F  
&=&    
\sum_{i=1}^n(-1)^i \; \F\left( \omega_V(v_i, z_i - z_{i+1}) \;    
 v_{i+1} \right)  
\nn
&+& 
 \omega_W \left(v_1, z_1 \right) \; \F (v_2, z_2; \ldots; v_n, z_n)   
\nn
 &+& (-1)^{n+1} 
 \omega_W(v_{n+1}, z_{n+1})  
\; \F(v_1, z_1; \ldots; v_n, z_n).
\end{eqnarray}
Note that it is assumed that the coboundary operator does not affect 
$dz_i^{\wt(v_i)}$-tensor multipliers in $\F$.  
In \cite{Huang} the following proposition is proved 
\begin{proposition}
\label{cochainprop}
The operator \eqref{hatdelta} obeys       
\begin{equation}
\label{conde}
{\delta}^n_m: C_m^n(V, \W)   
\to C_{m-1}^{n+1}(V, \W),   
\end{equation}  
\begin{equation}
\label{deltacondition}  
{\delta}^{n+1}_{m-1} \circ {\delta}^n_m=0,   
\end{equation} 
\begin{equation}
\label{hat-complex} 
0\longrightarrow C_m^0(V, \W) 
\stackrel{\delta^0_m}{\longrightarrow} C_{m-1}^1(V, \W) 
\stackrel{\delta^1_{m-1}}{\longrightarrow}\ldots 
\stackrel{\delta^{m-1}_1}{\longrightarrow}
 C_0^m(V, \W )\longrightarrow 0,  
\end{equation}
i.e., it provides the chain-cochain complex 
$\left(C_m^n(V, \W), \delta^n_m \right)$.  
\hfill $\qed$  
\end{proposition}
\section{Application: the product of $C^n_m(V, \W)$-spaces}
\label{productc}
In this Section we consider an application of the material of Section \ref{product} to   
 the spaces $C^n_m(V, \W)$ of a complex (Definition \ref{initialspace}) 
described in the previous Section.   
We introduce the product of elements of two spaces of the complex   
with the image in another space of the complex coherent with respect   
to the original differential \eqref{hatdelta}, and the symmetry property \eqref{shushu}.
 We prove the canonicity of the product, 
 and derive an analogue of Leibniz formula. 
\begin{definition}
For 
$\F(v_1, x_1; \ldots; v_k, x_k)  \in  C^k_m(V, \W)$,  and  
$\F(v'_1, y_1; \ldots; v'_n, y_n)  \in   C_{m'}^n(V, \W)$   
the product 
\begin{eqnarray*}
\F(v_1, x_1; \ldots; v_k, x_k) \cdot_\epsilon \F(v'_1, y_1; \ldots; v'_n, y_n)  
\mapsto \F\left( v_1, x_1; \ldots; v_k, x_k; v'_1, y_1; \ldots; v'_n, y_n; \epsilon \right),  
\end{eqnarray*}
 is a $\W_{x_1, \ldots, x_k; y_1, \ldots, y_n}$-valued rational form 
\begin{eqnarray}
\label{Z2n_pt_epsss0}
&& \langle w',  \F (v_1, x_1; 
\ldots; v_k, x_k; v'_1, y_1; \ldots; v'_n, y_n;  \epsilon) \rangle  
\nn
&& \qquad =\langle w',   
\F (v_1, x_1; \ldots; v_k, x_k) \cdot_\epsilon \F(v'_1, y_1; \ldots; v'_n, y_n) \rangle  
\nn 
& & \qquad  =  
\sum_{u\in V }  
 \langle w', Y^W_{WV}\left(   
\F (v_1, x_1;  \ldots; v_k, x_k), \zeta_1\right)\; u \rangle   
\nn
& &
\qquad   \qquad  \qquad  
 \langle w', Y^W_{WV}\left(  
\F(v'_1, y_1; \ldots; v'_n, y_n)  
 , \zeta_2 \right) \; \overline{u} \rangle,      
\end{eqnarray}
defined by \eqref{Z2n_pt_eps1q1}.   
\end{definition}
\begin{remark}
Let $t$ be the number of common vertex operators the mappings 
$\F(v_1, x_1$;  $\ldots$; $v_k, x_k) \in C^k_m(V, \W)$ and 
$\F(v'_1, y_1; \ldots; v'_n, y_n) \in C^n_{m'}(V, \W)$,  
are composable with. Similar to the case of common formal parameters, this case is 
separately treated with a decrease to $m+m'-t$ of the number of composable vertex operators.   
In what follows, we exclude this case from considerations. 
\end{remark}
The action of $\sigma \in S_{k+n}$ on the product 
$\F\left( v_1, x_1;  \ldots; v_k, x_k \right.$; $ v'_{k+1}$, $ y_1$; 
$ \ldots$; $\left. v'_n, y_n; 
\epsilon\right)$   
\eqref{Z2n_pt_epsss0}  
is given by \eqref{sigmaction}.   
We then have 
\begin{proposition}
\label{tolsto}
For $\F(v_1, x_1; \ldots; v_k, x_k) \in C_m^k(V, \W)$ and  
$\F(v'_1, y_1; \ldots; v'_n, y_n)\in C_{m'}^n(V, \W)$,  
the product $\F\left(v_1, x_1; \ldots; v_k, x_k; v'_1, y_1; \ldots; v'_n, y_n  
; \epsilon\right)$ \eqref{Z2n_pt_epsss0} 
belongs to the space $C^{k+n}_{m+m'}(V, \W)$, i.e.,  
\[
\cdot_\epsilon : C^k_m (V,\W) \times C_{m'}^n(V, \W) \to  C_{m+m'}^{k+n}(V, \W).   
\]
\end{proposition}
\begin{proof}
In Proposition \ref{derga} we proved that 
$\F\left(v_1, x_1; \ldots; v_k, x_k; v'_1, y_1; \ldots; v'_n, y_n  
; \epsilon\right) \in \W_{x_1;, \ldots, x_k;  y_1, \ldots, y_n}$.    
It is clear that  
\[
\cdot_\epsilon : C^k_.(V,\W) \times C_.^n(V, \W) \to  C_l^{k+n}(V, \W),  
\]
for some $l$. 
 First, we show that \eqref{shushu} for $\sigma \in S_{k+n}$,    
\begin{equation*} 
\sum_{\sigma\in J_{k+n; s}^{-1}}(-1)^{|\sigma|}
\F\left(v_{\sigma(1)}, x_{\sigma(1)}; \ldots; v_{\sigma(k)}, x_{\sigma(k)};  
v'_{\sigma(1)},  y_{\sigma(1)};  \ldots; v'_{\sigma(n)},  y_{\sigma(n)}\right)=0. 
\end{equation*}
 For arbitrary $w' \in W'$, we have 
\begin{eqnarray*}
&&
 \sum_{\sigma\in J_{k+n; s}^{-1}}(-1)^{|\sigma|}
 \langle w', 
\F\left(v_{\sigma(1)}, x_{\sigma(1)}; \ldots; v_{\sigma(k)}, x_{\sigma(k)};  
v'_{\sigma(1)},  y_{\sigma(1)};  \ldots; v'_{\sigma(n)},  y_{\sigma(n)})\right)
 \rangle 
\nn
&&
\nn
&&
=
\sum_{\sigma\in J_{k+n; s}^{-1}}(-1)^{|\sigma|}  \;  \sum\limits_{u\in V}
  \langle w', Y^W_{WV} 
\left(\F(v_{\sigma(1)}, x_{\sigma(1)}; \ldots; v_{\sigma(k)},  x_{\sigma(k)}), \zeta_1 \right) u  
\rangle 
\nn
&&
 \qquad \qquad \qquad  \langle w', Y^W_{WV} 
\left(\F(v'_{\sigma(1)}, y_{\sigma(1)}; \ldots; v'_{\sigma(n)},  y_{\sigma(n)}), \zeta_2  
\right) \overline{u} \rangle 
\nn
&&
=\sum\limits_{u\in V}
\sum_{\sigma\in J_{k+n; s}^{-1}}(-1)^{|\sigma|}   
  \langle w', e^{\zeta_1 L_W{(-1)}} \; Y_{W}(u, -\zeta_1) \;  
\F(v_{\sigma(1)}, x_{\sigma(1)}; \ldots; v_{\sigma(k)},  x_{\sigma(k)})   
\rangle 
\nn
&&
 \qquad \qquad \qquad  \langle w',  e^{\zeta_2 L_W{(-1)}} \; Y_W(\overline{u}, -\zeta_2) \;   
 \F(v'_{\sigma(1)}, y_{\sigma(1)}; \ldots; v'_{\sigma(n)},  y_{\sigma(n)})
 \rangle 
\nn
&&
=\sum\limits_{u\in V} 
  \langle w', e^{\zeta_1 L_W{(-1)}} \; Y_{W}(u, -\zeta_1) \; 
\sum_{\sigma\in J_{k; s}^{-1}}(-1)^{|\sigma|} 
\F(v_{\sigma(1)}, x_{\sigma(1)}; \ldots; v_{\sigma(k)},  x_{\sigma(k)})    
\rangle 
\nn
&&
 \qquad \qquad  \langle w',  e^{\zeta_2 L_W{(-1)}} \; Y_W(\overline{u}, -\zeta_2) \;    
 \F(v'_{\sigma(1)}, y_{\sigma(1)}; \ldots; v'_{\sigma(n)},  y_{\sigma(n)})
 \rangle 
\nn
&& 
+\sum\limits_{u\in V} 
  \langle w', e^{\zeta_1 L_W{(-1)}} \; Y_W(u, -\zeta_1) \; 
 \F(v_{\sigma(1)}, x_{\sigma(1)}; \ldots; v_{\sigma(k)},  x_{\sigma(k)})   
\rangle 
\nn
&&
 \qquad   \langle w',  e^{\zeta_2 L_W{(-1)}} \; Y_W(\overline{u}, -\zeta_2) \;    
\sum_{\sigma\in J_{n; s}^{-1}}(-1)^{|\sigma|}  
 \F(v'_{\sigma(1)}, y_{\sigma(1)}; \ldots; v'_{\sigma(n)},  y_{\sigma(n)}) 
 \rangle=0,  
\end{eqnarray*}
since, 
$J^{-1}_{k+n; s}= J^{-1}_{k;s} \times J^{-1}_{n;s}$,   
and 
  due to the fact that 
 $\F(v_1, x_1; \ldots; v_k,  x_k)$ 
 and 
$\F(v'_1, y_1 $; $ \ldots $; $ v'_n,  y_n)$  
satisfy \eqref{sigmaction}. 

Next, we show that 
$\F\left(v_1, x_1; \ldots; v_k,  x_k; v'_1, y_1; \ldots; v'_n,  y_n; \epsilon\right)$ 
 \eqref{Z2n_pt_epsss0}  
is composable with $m+m'$ vertex operators.  
Recall that 
$\F(v_1, x_1; \ldots; v_k,  x_k)$ 
is composable with $m$ vertex operators, and  
 $\F(v'_1, y_1; \ldots; v'_n,  y_n)$  
is composable with $m'$ vertex operators. 
For $\F(v_1, x_1; \ldots; v_k,  x_k)$ we have:  
1) Let $l_1, \dots, l_k \in \Z_+$ such that $l_1+\ldots +l_k= k+m$,  and  
$v_1, \dots, v_{k+m} \in V$, and arbitrary $w'\in W'$. Set   
 \begin{eqnarray}
\label{psii}
\Psi_i 
&
=
&
E^{(l_i)}_V (v_{k_1}, x_{k_1}- \zeta_i;   
 \ldots;  
v_{k_i}, x_{k_i}- \zeta_i; \one_V),     
\end{eqnarray}
where
\begin{eqnarray}
\label{ki}
 {k_1}={l_1+\ldots +l_{i-1}+1}, \quad  \ldots, \quad  k_i={l_1+\ldots +l_{i-1}+l_i},  
\end{eqnarray} 
for $i=1, \dots, k$.  
Then the series 
\begin{eqnarray}
\label{Inms}
\mathcal I^k_m(\F)=
\sum_{r_1, \dots, r_k \in \Z}\langle w', 
\F(P_{r_1}\Psi_1; \zeta_1;  
 \ldots; 
P_{r_k} \Psi_k, \zeta_k)  
\rangle,
\end{eqnarray} 
is absolutely convergent  when 
\begin{eqnarray}
\label{granizy1}
|x_{l_1+\ldots +l_{i-1}+p}-\zeta_i|  
+ |x_{l_1+\ldots +l_{j-1}+q}-\zeta_i|< |\zeta_i -\zeta_j|,  
\end{eqnarray} 
for $i$, $j=1, \dots, k$, $i\ne j$ and for $p=1,  
\dots,  l_i$ and $q=1, \dots, l_j$. 
There exist positive integers $N^k_m(v_i, v_j)$,  
depending only on $v_i$ and $v_j$ for $i, j=1, \dots, k$, $i\ne j$, such that 
the sum is analytically extended to a
rational function
in $(x_1, \dots, x_{k+m})$, 
 independent of $(\zeta_1, \dots, \zeta_k)$,   
with the only possible poles at 
$x_i=x_j$, of order less than or equal to  
$N^k_m(v_i, v_j)$, for $i$, $j=1, \dots, k$,  $i\ne j$.   

For $\F(v'_1, y_1; \ldots; v'_n,  y_n)$ we have:  

1') Let $l'_1, \dots, l'_n \in \Z_+$ such that $l'_1 +\ldots +l'_n= n+m'$,   
$v'_1, \dots, v_{n+m'} \in V$ and arbitrary $w'\in W'$.  
Set  
 \begin{eqnarray}
\label{psii}
\Psi'_{i'} 
&
=
&
E^{(l'_{i'})}_{V}(v'_{k'_1}, y_{k'_1}- \zeta'_{i'};  
 \ldots; 
v'_{k'_{i'}}, y_{k'_{i'}}- \zeta'_{i'} 
 ; \one_V),     
\end{eqnarray}
where
\begin{eqnarray}
\label{ki}
 {k'_1}={l'_1+\ldots +l'_{i'-1}+1}, 
\quad  \ldots, \quad  {k'_{i'}}={l'_1+\ldots +l'_{i'-1}+l'_{i'}},   
\end{eqnarray} 
for $i'=1, \dots, n$. 
  Then the series 
\begin{eqnarray}
\label{Jnms}
\mathcal I^n_{m'}(\F)=   
\sum_{r'_1, \dots, r'_n \in \Z}\langle w',  
\F(P_{r'_1}\Psi'_1; \zeta'_1;  
 \ldots; 
P_{r'_n} \Psi'_n, \zeta'_n)   
\rangle,
\end{eqnarray} 
is absolutely convergent when  
\begin{eqnarray}
\label{granizy2}
|y_{l'_1+\ldots +l'_{i'-1}+p'}-\zeta'_{i'}| 
+ |y_{l'_1 +\ldots +l'_{j'-1}+q'}-\zeta'_{i'}|< |\zeta'_{i'}
-\zeta'_{j'}|, 
\end{eqnarray} 
for $i'$, $j'=1, \dots, n$, $i'\ne j'$ and for $p'=1, 
\dots, l'_i$ and $q'=1, \dots, l'_j$. 
There exist positive integers $N^n_{m'}(v'_{i'}, v'_{j'})$,   
depending only on $v'_{i'}$ and $v'_{j'}$ for $i$, $j=1, \dots, n$, $i'\ne j'$, such that 
the sum is analytically extended to a 
rational function 
in $(y_1, \dots, y_{n+m'})$,    
 independent of $(\zeta'_1, \dots, \zeta'_n)$,    
with the only possible poles at 
$y_{i'}=y_{j'}$, of order less than or equal to 
$N^{n}_{m'}(v'_{i'}, v'_{j'})$, for $i'$, $j'=1, \dots, n$,  $i'\ne j'$.  

Now let us consider the first condition of Definition \ref{composabilitydef} of composability
for the product  
\eqref{Z2n_pt_epsss0} of 
$\F(v_1, x_1; \ldots; v_k,  x_k)$  
 and 
 $\F(v'_1, y_1; \ldots; v'_n,  y_n)$ with a number of vertex operators.   
Then we obtain for  
$\F\left(v_1, x_1; \ldots; v_k, x_k; v'_1, y_1; \ldots; v'_n, y_n ; \epsilon\right)$ 
the following.  
 We redefine the notations for the set 
\begin{eqnarray*}
  && (v''_{1}, \ldots, v''_k; v''_{k+1}, \ldots, v''_{k+m}; v''_{k+m+1}, \dots, v''_{k+n+m+m'};
 v_{n+1}, \ldots, v'_{n+m'})
\nn
&&
 \qquad \qquad 
=(v_1, \ldots, v_k; v_{k+1}, \ldots, v_{k+m}; v'_1, \dots, v'_n; 
 v'_{n+1}, \ldots, v'_{n+m'}), 
\nn
&&
(z_1, \ldots, z_k; z_{k+1},  \dots, z_{k+n}) =  
  (x_1, \ldots, x_k; y_1, \ldots, y_n),  
\end{eqnarray*}
 of vertex algebra $V$ elements.  
Introduce $l''_1, \dots, l''_{k+n} \in \Z_+$,      
 such that $l''_1+\ldots +l''_{k+n}= k+n+m+m'$.  
Define  
 \begin{eqnarray}
\label{psiinew}
\Psi''_i
&
=
&
E^{(l''_{i''})}_{V}(v''_{k''_1}, z_{k''_1}- \zeta''_{i''};  
 \ldots; 
v''_{k''_{i''}}, z_{k''_{i''}}- \zeta''_{i''} ; \one_V),      
\end{eqnarray}
where
\begin{eqnarray}
\label{ki}
 {k''_1}={l''_1+\ldots +l''_{i''-1}+1}, \quad  \ldots, \quad  {k''_{i''}}={l''_{1}+\ldots 
+l''_{i''-1}+l''_{i''}},   
\end{eqnarray} 
for $i''=1, \dots, k+n$, 
and we take 
\[
(\zeta''_1, \ldots, \zeta''_{k+n})= (\zeta_1, \ldots, \zeta_{k}; \zeta'_1, \ldots, \zeta'_n).  
\]  
Then we consider 
\begin{eqnarray}
\label{Inmdvadva}
 && \mathcal I^{k+n}_{m+m'}(\F)=  
\sum_{r''_{1}, \dots, r''_{k+n}\in \Z}
 \langle w',  
\F(P_{r''_{1}}\Psi''_{1}; \zeta''_1; 
 \ldots; 
P_{r''_{k+n}} \Psi''_{k+n}, \zeta''_{k+n})  
\rangle,
\end{eqnarray} 
and prove it is absolutely convergent with some conditions. 

 The condition    
\begin{eqnarray}
\label{granizy1000000}
&&
|z_{l''_1+\ldots +l''_{i-1}+p''}-\zeta''_i|  
+ |z_{l''_1+\ldots +l''_{j-1}+q''}-\zeta''_i|< |\zeta''_i -\zeta''_j|,   
\end{eqnarray} 
of absolute convergence for \eqref{Inmdvadva} 
for $i''$, $j''=1, \dots, k+n$, $i\ne j$ and for 
$p''=1, \dots,  l''_i$ and $q''=1, \dots, l''_j$, 
follows from the conditions \eqref{granizy1} and \eqref{granizy2}.   
 The action of $e^{\zeta L_W{(-1)} } \;  Y_W(.,.)$, $a=1$, $2$, in 
\[
\langle w',  e^{\zeta_1 L_W{(-1)} } \; Y_W(u, -\zeta)\sum_{r_1, \dots, r_k \in \Z}  
\F(P_{r_1}\Psi_1; \zeta_1;  
 \ldots; 
P_{r_k} \Psi_k, \zeta_k)   \rangle,  
\]
\[
\langle w',  e^{\zeta_2 L_W{(-1)} } 
\; Y_W(\overline{u}, -\widetilde{\zeta}) \sum_{r'_1, \dots, r'_n\in \Z}  
\F(P_{r'_1}\Psi'_1; \zeta_1;  
 \ldots; 
P_{r'_k} \Psi'_n, \zeta'_n)  \rangle,   
\]
 does not affect the absolute convergence of \eqref{Inms} and \eqref{Jnms}. 
 We obtain 
\begin{eqnarray*} 
 && \left|\mathcal I^{k+n}_{m+m'}(\F)\right|=  
\nn
&&
=\left|\sum_{r''_1, \dots, r''_{k+n}\in \Z}
 \langle w',  
\F(P_{r''_1}\Psi''_1; \zeta''_1;  
 \ldots; 
P_{r''_{k+n}} \Psi''_{k+n}, \zeta''_{k+n})  
\rangle\right| 
\nn
&&
=\left| \sum\limits_{u\in V} 
\langle w',  Y^W_{VW}(\sum_{r_1, \dots, r_k \in \Z}  
\F(P_{r_1} \Psi_1; \zeta_1;  
 \ldots; 
P_{r_k} \Psi_k, \zeta_k), \zeta) u  \rangle \right.
\nn
&&
\left. 
\qquad \qquad \langle w',  Y^W_{VW}(\sum_{r'_1, \dots, r'_n \in \Z}  
\F(P_{r'_1}\Psi'_1; \zeta'_1;  
 \ldots; 
P_{r'_n} \Psi'_n, \zeta'_n), \widetilde{\zeta}) \overline{u}
\rangle \right|
\nn
&&
 \qquad \qquad \qquad 
\le \left|\mathcal I^k_m(\F)\right| \; \left|\mathcal I^n_{m'}  (\F)\right|.   
\end{eqnarray*} 
Thus, we infer that \eqref{Inmdvadva}
is absolutely convergent. 
Recall that 
the maximal orders of possible poles of \eqref{Inmdvadva} are $N^k_m(v_i, v_j)$, 
$N^n_{m'}(v'_{i'}, v'_{j'})$ 
 at $x_i=x_j$, $y_{i'}=y_{j'}$.  
From the last expression we infer that 
there exist 
 positive integers $N^{k+n}_{m+m'}(v''_{i''}, v''_{j''})$ 
 for $i$, $j=1, \dots, k$, $i\ne j$,   $i'$, $j'=1, \dots, n$, $i'\ne j$,   
depending only on $v''_{i''}$ and $v''_{j''}$ for $i''$,  $j''=1, \dots, k+n$, $i''\ne j''$
such that  
 the series \eqref{Inmdvadva} 
can be analytically extended to a 
rational function
in $(x_1, \dots, x_k; y_1, \ldots, y_n)$,     
 independent of $(\zeta''_1, \dots, \zeta''_{k+n})$,    
with extra 
 possible poles at  
 and $x_i=y_j$,  of order less than or equal to  
$N^{k+n}_{m+m'}(v''_{i''}, v''_{j''})$, for $i''$, $j''=1, \dots, n$,  $i''\ne j''$.   

Let us proceed with the second condition of composability. 
For $\F(v_1, x_1; \ldots; v_k, x_k)  \in  C^k_m(V, \W)$, and  
$(v_1, \ldots, 
v_{k+m}) \in V$, $(x_1, \ldots, x_{k+m})\in \C$,      
 we have
 
2) 
For arbitrary  $w'\in W'$, the series 
\begin{eqnarray}
\label{Jnm2}
\mathcal J^k_m(\F)=  
\sum_{q\in \C}\langle w', 
E^{(m)}_W \Big(v_1, x_1;  \ldots;   
v_m, x_m;   
P_q ( \F(v_{m+1}, x_{m+1}; \ldots; 
v_{m+k}, x_{m+k} \Big)\rangle,  
\nn
&&
\end{eqnarray}
is absolutely convergent when 
\begin{eqnarray}
\label{granizy2}
x_i \ne x_j, \quad i\ne j, \quad  
\nn
|x_i|>|x_{k'}|>0,  
\end{eqnarray}
 for $i=1, \dots, m$, and $k'=m+1, \dots, k+m$,
 and the sum can be analytically extended to a rational function  
in $(x_1, \dots, x_{k+m})$ with the only possible poles at 
$x_i=x_j$, of orders less than or equal to 
$N^k_m(v_i, v_j)$, for $i$, $j=1, \dots, k$, $i\ne j$.   

2')  For 
 $\F(v'_1, y_1; \ldots; v'_n, y_n)  \in   C_{m'}^n (V, \W)$,  
$(v'_1, \ldots,  
v'_{n+m'})\in V$,  and 
$(y_1, \ldots,  
y_{n+m'})\in \C$,  
 the series  
\begin{eqnarray}
\label{Jnm}
&& \mathcal J^n_{m'}(\F)=   
\sum_{q\in \C}\langle w', E^{(m')}_W \Big(v'_1, y_1; \ldots; 
 v'_{m'}, y_{m'};  
\nn 
&& \quad\quad\quad  
P_q( \F(v'_{m'+1}, y_{m'+1}; \ldots; v'_{m'+n}, y_{m'+n}) )\Big)\rangle,  
\end{eqnarray}
is absolutely convergent when 
\begin{eqnarray}
\label{granizy2}
y_{i'}\ne y_{j'}, \quad i'\ne j', \quad 
\nn
|y_{i'}|>|y_{k''}|>0, 
\end{eqnarray}
 for $i'=1, \dots, m'$, and $k''=m'+1, \dots, n +m'$, 
and the sum can be analytically extended  
to a rational function  
in $(y_1, \ldots, y_{n+m'})$ with the only possible poles at  
$y_{i'}=y_{j'}$, of orders less than or equal to 
$N^n_{m'}(v'_{i'}, v'_{j'})$, for $i'$,  $j'=1, \dots, n$, $i'\ne j'$.   

2'') 
Thus, for the product \eqref{Z2n_pt_epsss0} we obtain 
  $(v''_1, \dots, v''_{k+n +m+m'})\in V$,   
and  
$(z_1, \ldots $ , $ z_{k+n+m+m'} )\in \C$,  
we find  
positive integers   
$N^{k+n}_{m+m'}(v'_i, v'_j)$,    
 depending only on $v'_i$ and 
$v''_j$, for $i''$, $j''=1, \dots, k+n$, $i''\ne j''$, such that for arbitrary $w'\in W'$.   
First we note 
\begin{lemma}
\label{obvlem}
\begin{eqnarray*}
&&
\sum_{q\in \C}\langle w', E^{(m+m')}_W \Big( 
v''_1, z_1; \ldots;   
v''_{m+m'}, z_{m+m'};     
\nn
&&
 \qquad \qquad  \qquad   
P_q \Big( \F(v''_{m+m'+1}, z_{m+m'+1}; \ldots; v''_{m+m'+k+n}, z_{m+m'+k+n}  
\Big) \Big) \rangle 
\nn
&&
=
\sum_{u\in V }
\langle w',  E^{(m)}_W \Big( 
v_{k+1}, x_{k+1}; \ldots; 
v_{k+m}, x_{k+m};   
\nn
&& 
\qquad \qquad  \qquad \qquad  \qquad \qquad  P_q \Big(     
  Y^W_{WV}\left(   
\F (v_1, x_1;  \ldots; v_k, x_k), \zeta_1\right)\; u \Big) \Big)\rangle  
\nn
&& 
\qquad   \qquad  
 \langle w',  E^{(m')}_W \Big( 
v'_{n+1}, y_{n+1}; \ldots;  
v'_{n+m'}, y_{n+m'};   
\nn
&&
  \qquad \qquad \qquad \qquad  \qquad \qquad  
P_q \Big( Y^W_{WV}\left(  
\F (v'_1, y_1; \ldots; v'_n, y_n) , \zeta_2 \right) 
 \; \overline{u} \Big) \Big) \rangle.    
\end{eqnarray*}
\end{lemma}
\begin{proof}
Consider
\begin{eqnarray*}
&& \sum_{u\in V } 
\langle w',  E^{(m+m')}_W \Big( 
v''_1, z_1; \ldots;  
v''_{m+m'}, z_{m+m'};  
\nn
&& 
\qquad \qquad  P_q \Big(   
  Y^W_{WV}\left(   
\F (v''_{m+m'+1}, z_{m+m'+1};  
\ldots; v''_{m+m' +k}, z_{m+m'+k}), \zeta_1\right)\; u \Big) \Big) \rangle   
\nn
& & 
\qquad   \qquad 
 \langle w',  E^{(m+m')}_W \Big( 
v''_1, z_1; \ldots; v''_{m+m'}, z_{m+m'};  
\nn
&&
  \qquad \qquad P_q \Big( Y^W_{WV}\left(  
\F (v''_{m+m'+k+1}, z_{m+m'+k+1}; \ldots; \right.  
\nn
&&
      \qquad \qquad  \qquad \qquad    \qquad \qquad \qquad \qquad \left. 
 v''_{m+m'+k+n}, z_{m+m'+k+n}) , \zeta_2\right) \; \overline{u} \Big) \Big) \rangle    
\end{eqnarray*}
\begin{eqnarray*}
&&
= \sum_{q\in \C}  \sum_{u\in V }  \langle w', E^{(m+m')}_W \Big( 
v''_1, z_1; \ldots; 
v''_{m+m'}, z_{m+m'};  
\nn
&&
\qquad   
  P_q \Big( e^{\zeta_1 L_W{(-1)} }\; Y_W  \left( u, -\zeta_1 )  
\; \F (v''_{m+m'+1}, z_{m+m'+1};  \ldots; v''_{m+m'+ k }, z_{m+m'+k})  \right) \Big) \rangle  
\nn
& & 
\quad    
\langle w',  E^{(m+m')}_W \Big( 
v''_1, z_1; \ldots; 
v''_{m+m'}, z_{m+m'};  
\nn
&&
  P_q \Big( e^{\zeta_2 L_W{(-1)}}\; Y_W   
\left(  \overline{u}, -\zeta_2  )\;  
\F (v''_{m+m'+k+1}, z_{m+m'+k+1}; \ldots; \right. 
\nn
&&
   \qquad \qquad     \qquad \qquad  \qquad \qquad    \qquad \qquad \qquad \qquad \left.
 v''_{m+m'+k+n}, z_{m+m'+k+n})   
\right)  
\Big)\rangle.  
\end{eqnarray*}
The action of exponentials $e^{\zeta_a L_W{(-1)} }$, $a=1$, $2$, 
of the differential operator $L_W{(-1)}$,    
and $W$-module vertex operators $Y_W \left(u, -\zeta_1 \right)$,  
$Y_W \left( u, -\zeta_2 \right)$   
 shifts the grading index $q$ of $W_q$-subspaces by $\alpha \in \C$  
which can be later rescaled  
to $q$.  
 Thus, we can rewrite the last expression as  
\begin{eqnarray*}
&&
= \sum_{q\in \C}  \sum_{u\in V }  \langle w', E^{(m+m')}_W \Big( 
v''_1, z_1; \ldots; 
v''_{m+m'}, z_{m+m'};  
\nn
&&
\qquad    
  e^{\zeta_1 L{_W(-1)} }\; Y_W  \left ( u, -\zeta_1 \big)  
\;P_{q+\alpha}\Big(   \F (v''_{m+m'+1}, z_{m+m'+1}; 
 \ldots; v''_{m+m'+ k }, z_{m+m'+k})  \right) \Big) \rangle 
\nn
& & 
\langle w',  E^{(m+m')}_W \Big( v''_1, z_1; \ldots;  v''_{m+m'}, z_{m+m'};   
\nn
&&
\quad   
  e^{\zeta_2 L_W{(-1)} }\; Y_W  \left(  \overline{u}, -\zeta_2  \Big)\;    
 P_{q+\alpha}\Big(  \F 
(v''_{m+m'+k+1}, z_{m+m'+k+1}; \ldots; v''_{m+m'+k+n}, z_{m+m'+k+n})  \right)  
\Big)\rangle  
\end{eqnarray*}
\begin{eqnarray*}
&&
= \sum_{q\in \C}  \sum_{u\in V }  \langle w', E^{(m+m')}_W \Big( 
v''_1, z_1; \ldots; v''_{m+m'}, z_{m+m'};   
\nn
&&
Y^W_{WV} \left( 
 P_{q+\alpha}\Big(   \F (v''_{m+m'+1}, z_{m+m'+1};  \ldots; v''_{m+m'+ k }, z_{m+m'+k}) 
\Big), \zeta_1 \Big)\; u \rangle \right.  
\nn
&&
\langle w',  E^{(m+m')}_W \Big( v''_1, z_1; \ldots;  v''_{m+m'}, z_{m+m'};   
\nn
& &  
Y^W_{WV}    
\left(  
 P_{q+\alpha}\Big(  \F 
(v''_{m+m'+k+1}, z_{m+m'+k+1}; \ldots; 
v''_{m+m'+k+n}, z_{m+m'+k+n})  , -\zeta_2) \; \overline{u} \right)\rangle   
\end{eqnarray*}
\begin{eqnarray*}
&&
= \sum_{q\in \C}  \sum_{\widetilde{w}\in W }   \langle w', E^{(m+m')}_{W} \Big(
v''_1, z_1; \ldots;  
v''_{m+m'}, z_{m+m'};  \widetilde{w} \Big) \rangle  
\nn
&&
\qquad   
\sum_{u\in V } 
\langle w',  Y^W_{WV}   
\left( 
\; P_{q+\alpha}\Big(   \F (v''_{m+m'+1}, z_{m+m'+1};  \ldots; v''_{m+m'+ k }, z_{m+m'+k}), 
-\zeta_1)\; u  \right)  \Big)  \rangle  
\nn
& & 
\quad    
\langle \widetilde{w}', E^{(m+m')}_W \Big( v''_1, z_1; \ldots;  
v''_{m+m'}, z_{m+m'};  \widetilde{w} \Big) \rangle   
\nn
&&
\langle w', Y^W_{WV} \left( P_{q+\alpha}
\Big(  \F 
(v''_{m+m'+k+1}, z_{m+m'+k+1}; \ldots; v''_{m+m'+k+n}, z_{m+m'+k+n})  , -\zeta_2) \; 
\overline{u}  \right)  \Big)\rangle     
\end{eqnarray*}
\begin{eqnarray*}
&&
= \sum_{q\in \C}  
 \langle w', E^{(m+m')}_W \Big( v''_1, z_1; \ldots;  v''_{m+m'}, z_{m+m'};   
\nn
&&
\qquad    
 P_{q+\alpha}\Big(   \F (v''_{m+m'+1}, z_{m+m'+1};  \ldots; v''_{m+m'+ k }, z_{m+m'+k}; 
\nn
& & 
\qquad   \qquad  
v''_{m+m'+k+1}, z_{m+m'+k+1}; \ldots; v''_{m+m'+k+n}, z_{m+m'+k+n}) \Big) \rangle.    
\end{eqnarray*}
Now note that, according to Proposition \ref{pupa},
 as an element of $\W_{z_1, \ldots,  z_{k+n+m+m'}}$ 
\begin{eqnarray}
\label{svoloch}
&& \langle w', E^{(m+m')}_W \Big( 
v''_1, z_1; \ldots;  v''_{m+m'}, z_{m+m'};   
\nn
&&
\qquad   
 P_{q+\alpha}\Big(   \F (v''_{m+m'+1}, z_{m+m'+1};  \ldots; v''_{m+m'+ k }, z_{m+m'+k}; 
\nn
&&  
\qquad   \qquad  
v''_{m+m'+k+1}, z_{m+m'+k+1}; \ldots; v''_{m+m'+k+n}, z_{m+m'+k+n}) \Big) 
\rangle,  
\end{eqnarray}
is invariant with respect to the action of $\sigma \in S_{k+n+m+m'}$. 
 Thus we are able to use this invariance to show that \eqref{svoloch} is reduced to 
\begin{eqnarray*} 
&& \langle w', E^{(m+m')}_W \Big(  
v''_{k+1}, z_{k+1}; \ldots; v''_{k+1+m}, z_{k+1+m};    
v''_{n+1}, z_{n+1}; \ldots; v''_{n+1+m'}, z_{n+1+m'};   
\nn
&&
\qquad   
 P_{q+\alpha}\Big(   \F (v''_1, z_1;  \ldots; v''_k, z_k;  
v''_{k+1}, z_{k+1}; \ldots; v''_{k+n}, z_{k+n}) \Big) \Big)
\rangle 
\nn
&&
=\langle w', E^{(m+m')}_W \Big(  
v_{k+1}, x_{k+1};  \ldots; v_{k+1+m}, x_{k+1+m}; 
v'_{n+1}, y_{n+1};  \ldots; v'_{n+1+m'}, y_{n+1+m'};  
\nn
&&
\qquad    
 P_{q+\alpha}\Big(   \F (v_1, x_1;  \ldots; v_k, x_k;    
v'_1, y_1; \ldots; v'_n, y_n) \Big) 
\rangle.  
\end{eqnarray*}
Similarly, since  
\begin{eqnarray*}
&& 
\langle w',  E^{(m)}_W \Big( 
v''_1, z_1; \ldots; 
v''_{m+m'}, z_{m+m'};  
\nn
&& 
\qquad \qquad  P_q \Big(    
  Y^W_{WV}\left(   
\F (v''_{m+m'+1}, z_{m+m'+1};  \ldots; 
v''_{m+m' +k}, z_{m+m'+k}), \zeta_1\right)\; u \Big) \Big)\rangle,   
\end{eqnarray*} 
\begin{eqnarray*}
& & 
 \langle w',  E^{(m')}_W \Big( 
v''_1, z_1; \ldots; v''_{m+m'}, z_{m+m'};  
\nn
&&
  \qquad 
P_q\Big( Y^W_{WV}\left( \F  
(v''_{m+m'+k+1}, z_{m+m'+k+1}; \ldots; v''_{m+m'+k+n}, z_{m+m'+k+n}) , \zeta_2\right) \;  
\overline{u} \Big) \Big) \rangle.    
\end{eqnarray*}
correspond to elements of $\W_{z_1, \ldots, z_{m+m'+k}}$ and  
$\W_{z_{m+m'+k+1}, \ldots, z_{m+m'+k+n}}$, 
we use Proposition \ref{pupa} again and obtain 
\begin{eqnarray*}
&& 
\langle w',  E^{(m)}_W \Big( 
v_{k+1}, x_{k+1}; \ldots;   
v_{k+m}, x_{k+m};  
 P_q  \Big(   
  Y^W_{WV}\left(   
\F (v_1, x_1;  \ldots; v_k, x_k), \zeta_1 \right)\; u \Big) \Big)\rangle 
\end{eqnarray*} 
\begin{eqnarray*}
& &
 \langle w',  E^{(m')}_W \Big( 
v'_{n+1}, y_{n+1}; \ldots;  
v'_{n+m'}, y_{n+m'}; 
 P_q \Big(    Y^W_{WV}\left(  
\F (v'_1, y_1; \ldots; v'_n, y_n) , \zeta_2 \right) \; \overline{u} \Big) \Big) \rangle,   
\end{eqnarray*}
correspondingly. 
Thus, the assertion of Lemma follows. 
\end{proof}
Under conditions
\begin{eqnarray}
\label{granizy2}
z_{i''}\ne z_{j''}, \quad i''\ne j'', \quad 
\nn
|z_{i''}|>|z_{k'''}|>0, 
\end{eqnarray}
 for $i''=1, \dots, m+m'$, and $k'''=m+m'+1, \dots, m+m'+ k+n$, 
let us introduce 
\begin{eqnarray}
\label{perda}
&&
 \mathcal J^{k+n}_{m+m'}(\F) = \sum_{q\in \C} 
\langle w', E^{(m+m')}_W \Big(
v''_1, z_1; \ldots;  
v''_{m+m'}, z_{m+m'};  
\nn
&& 
\qquad  
 P_q \Big( \F( v''_{m+m'+1}, z_{m+m'+1}; \ldots; v''_{m+m'+k+n}, z_{m+m'+k+n})  
; \epsilon)\Big)\rangle.   
\end{eqnarray}
Using Lemma \ref{obvlem} we obtain 
\begin{eqnarray*}
&&
|\mathcal J^{k+n}_{m+m'}(\F) | 
\nn
&&
= \left| \sum_{q\in \C}\langle w', E^{(m+m')}_W \Big( 
v''_1, z_1; \ldots;  v''_{m+m'}, z_{m+m'};  \right.  
\nn
&&
\left.  
\qquad P_q \Big( \F( v''_{m+m'+1}, z_{m+m'+1}; \ldots; v''_{m+m'+k+n}, z_{m+m'+k+n}) 
; \epsilon)\Big)\rangle \right| 
\end{eqnarray*}
\begin{eqnarray*}
&&
= \left| \sum_{q\in \C} \sum_{u\in V } 
\langle w',  E^{(m)}_W \Big( v_{k+1}, x_{k+1}; \ldots;  v_{k+m}, x_{k+m};   
\right.
\nn
&& 
\qquad \qquad  \qquad \qquad \qquad \qquad 
P_q \Big( Y^W_{WV}\left(   
\F (v_1, x_1;  \ldots; v_k, x_k), \zeta_1\right)\; u \Big) \Big)\rangle    
\nn
& &
\qquad   \qquad  
 \langle w',  E^{(m')}_W \Big( v'_{n+1}, y_{n+1}; \ldots;  v'_{n+m'}, y_{n+m'};   
\nn
&&
\left. 
\qquad \qquad 
  \qquad \qquad \qquad \qquad 
P_q \Big( Y^W_{WV}\left( \F 
(v'_1, y_1; \ldots; v'_n, y_n) , \zeta_2\right) \; \overline{u} \Big) \Big) \rangle  \right| 
\nn
&&
\le \left| 
\mathcal J^k_m (\F) \right| \; \left|  \mathcal J^n_{m'}(\F)\right|,    
\end{eqnarray*}
where we have used 
the invariance of \eqref{Z2n_pt_epsss0} with respect to  
$\sigma \in S_{m+m'+k+n}$. 
According to Proposition \ref{correl-fn} 
$\mathcal J^k_m(\F)$ and $\mathcal J^n_{m'}(\F)$ in the last expression  
are absolute convergent. 
Thus, we infer that $\mathcal J^{k+n}_{m+m'}(\F)$  
is absolutely convergent, and 
the sum \eqref{Inmdvadva} 
 is analytically extendable to a  rational function  
in $(z_1, \dots, z_{k+n+m+m'})$ with the only possible poles at 
$x_i=x_j$, $y_{i'}=y_{j'}$, and   
at $x_i=y_{j'}$, i.e., the only possible poles at 
$z_{i''}=z_{j''}$, of orders less than or equal to 
$N^{k+n}_{m+m'}(v''_{i''}, v''_{j''})$, 
for $i''$, $j''=1, \dots, k'''$, $i''\ne j''$.   
This finishes the proof of the proposition. 
\end{proof}
Now we prove the following   
\begin{corollary}
For $\F(v_1, x_1;  \ldots; v_k, x_k) \in C_m^k(V, \W)$ and  
$\F(v'_1, y_1; \ldots; v'_n, y_n) \in C_{m'}^n(V $, $ \W)$,  
the product 
\begin{eqnarray}
\label{posta}
&& \F \left( v_1, x_1;  \ldots; v_k, x_k; v'_1, y_1; \ldots; v'_n, y_n;  
\epsilon \right)  
\nn
&& \qquad \qquad 
=
\F(v_1, x_1;  \ldots; v_k, x_k) \cdot_\epsilon \F (v'_1, y_1; \ldots; v'_n, y_n), 
\end{eqnarray} 
is canonical with respect to the action  
\begin{eqnarray}
&&(x_1, \ldots, x_k; y_1, \ldots, y_n) \mapsto (x'_1, \ldots, x'_k; y'_1, \ldots, y'_n) 
\nn
&&
=(\rho_1( x_1, \ldots, x_k; y_1, \ldots, y_n), \ldots, 
\rho_{k+n}( x_1, \ldots, x_k; y_1, \ldots, y_n)),     
\end{eqnarray}
 of elements of the group 
 ${\rm Aut}_{x_1, \ldots, x_k; y_1, \ldots, y_n} \; \Oo^{(k+n)}$.  
\end{corollary}
\begin{proof} 
In Subsection \ref{properties} we have proved that the product \eqref{Z2n_pt_eps1q1} 
 belongs to $W_{x_1, \ldots, x_k; y_1, \ldots, y_n}$, and is invariant with  
respect to the group ${\rm Aut}_{x_1, \ldots, x_k; y_1, \ldots, y_n}\; \Oo^{(k+n)}$.   
Similar as in the proof of Proposition \ref{nezc},   
 vertex operators 
$\omega_V(v_i, x_i)$, $1\le i \le m$, composable with $\F(v_1, x_1; \ldots; v_k, x_k)$,   
 and vertex operators $\omega_V(v_j, y_j)$, $1 \le j \le m'$, composable with   
$\F(v'_1, y_1; \ldots; v'_n, y_n)$, are also invariant with respect to   
 $(\rho_1 ( x_1, \ldots, x_k; y_1, \ldots, y_n), \ldots $,  
$ \rho_{k+n}( x_1, \ldots, x_k; y_1, \ldots, y_n)) \in 
 {\rm Aut}_{x_1, \ldots, x_k; y_1, \ldots, y_n}\; \Oo^{(k+n)}$.    
\end{proof}
Since the product of $\F(v_1, x_1$;  $ \ldots $; $v_k, x_k) \in C^k_m(V, \W)$ and   
$\F(v'_1, y_1; \ldots; v'_n, y_n) \in C^n_{m'}(V, \W)$ results in  
an element of $C^{k+n}_{m+m'}(V, \W)$, then,  
similar to Proposition \ref{tudaty} \cite{Huang}, the following 
corollary follows directly from Proposition \eqref{tolsto} and Definition \ref{sprod}: 
\begin{corollary} 
For the spaces $\W_{x_1, \ldots, x_k}$ and $\W_{y_1, \ldots, y_n}$
 with the product \eqref{Z2n_pt_eps1q1} $\F \in \W_{x_1, \ldots, x_k; y_1, \ldots, y_n}$,   
the subspace of  
 $\W_{x_1, \ldots, x_k; y_1, \ldots, y_n})$    
consisting of maps  
having the $L_W(-1)$-derivative property, having the $L_V(0)$-conjugation property 
or being composable with $m$ vertex operators is invariant under the 
action of $S_{k+n}$.
\end{corollary}
Finally, we have the following
\begin{corollary}
\label{functionformprop} 
For a fixed set $(v_1, \ldots v_k; v_{k+1}, \ldots, v_{k+n}) \in V$
 of vertex algebra elements, and fixed $k+n$, and $m+m'$,    
 the $\epsilon$-product  
 $\F(v_1, z_1; \ldots; v_k, z_k; v_{k+1}, z_{k+1}; \ldots $ ; $ v_{k+n}, y_{k+n}; \epsilon)$,  
\[
\cdot_{\epsilon}: C^k_m(V, \W) \times C^n_{m'}(V, \W) \rightarrow C^{k+n}_{m+m'}(V, \W), 
\]
of the spaces $C^k_m(V, \W)$ and $C^n_{m'}(V, \W)$,  
for all choices 
of $k$, $n$, $m$, $m'\ge 0$, 
is the same element of $C^{k+n}_{m+m'}(V, \W)$ 
for all possible $k \ge 0$. 
\end{corollary}
\begin{proof}
In Proposition \ref{derga} we have proved that 
the result of the maps belongs to $W_{x_1, \ldots, x_k; y_1, \ldots, y_n}$,  
for all $k$, $n \ge 0$, and fixed $k+n$.     
As in the proof of Proposition \ref{tolsto},  
 by checking conditions  for the  forms \eqref{Inms} and \eqref{Jnms},  
we see that, by Proposition \ref{comp-assoc}, the product 
$\F(v_1, x_1; \ldots; v_k, x_k; v'_1, y_1; \ldots; v'_n, y_n)$
 is composable with fixed $m+m'$ vertex operators for evey $0 \le l \le k$.   
\end{proof}
\subsection{Coboundary operator acting on the product space} 
In Proposition \ref{tolsto} we proved that the  product \eqref{Z2n_pt_epsss0} of elements of 
  spaces  $C_m^k(V, \W)$ and $C_{m'}^n(V, \W)$ belongs to $C^{k+n}_{m+m'}(V, \W)$.  
Thus, the product admits the action  
of the differential operator $\delta^{k+n}_{m+m'}$ defined in  
\eqref{hatdelta}. 
 The co-boundary operator \eqref{hatdelta} 
 has a variation of Leibniz law with respect to the product  
\eqref{Z2n_pt_epsss0}. Indeed, we state here  
\begin{proposition}
For $\F(v_1, x_1;  \ldots; v_k, x_k) \in C_m^k(V, \W)$  
and 
$\F(v'_1, y_1; \ldots; v'_n, y_n) \in C_{m'}^n(V, \W)$,   
the action of $\delta_{m + m'}^{k + n}$ on their product \eqref{Z2n_pt_epsss0} is given by 
\begin{eqnarray}
\label{leibniz}
&& \delta_{m + m'}^{k + n} \left(  \F (v_1, x_1;  \ldots; v_k, x_k)  
 \cdot_\epsilon \F (v'_1, y_1; \ldots; v'_n, y_n) \right)  
\nn
&&
 \qquad =  
\left( \delta^k_m \F (v_1, x_1;  \ldots; v_k, x_k) \right)  
\cdot_\epsilon \F (v'_1, y_1; \ldots; v'_n, y_n)   
\nn
&&
\qquad  + (-1)^k 
\F (v_1, x_1;  \ldots; v_k, x_k) \cdot_{\epsilon}   \left( \delta^n_{m'}   
\F(v'_1, y_1; \ldots; v'_n, y_n)  \right).  
\end{eqnarray}
\end{proposition}
\begin{remark}
Checking \eqref{hatdelta} we see that   
an extra arbitrary vertex algebra element $v_{n+1} \in V$, as well as a corresponding  
 extra arbitrary formal parameter $z_{n+1}$ appear
 as a result of the action of $\delta^n_m$ on  
$\F \in C^n_m(V, \W)$ mapping it to $C^{n+1}_{m-1}(V, \W)$. 
In application to the $\epsilon$-product \eqref{Z2n_pt_epsss0}
 these extra arbitrary elements are involved in the  
definition of the action of $\delta_{m + m'}^{k + n}$ on 
 $\F (v_1, x_1;  \ldots; v_k, x_k) 
 \cdot_{\epsilon} \F (v'_1, y_1; \ldots; v'_n, y_n)$.   
\end{remark}
\begin{proof}
According to \eqref{hatdelta}, the action of 
$\delta_{m + m'}^{k + n}$ on 
$\F(v_1, x_1; \ldots; v_k, x_k; v'_1, y_1; \ldots; v'_k, y_n; \epsilon)$  
is given by  
\begin{eqnarray*}
&&   
\langle w', 
 \delta_{m + m'}^{k + n} \;
\F(v_1, x_1; \ldots; v_k, x_k; v'_1, y_1; \ldots; v'_n, y_n ; \epsilon) \rangle  
\nn 
&& \quad =
\langle w', \sum_{i=1}^k(-1)^i \;   
\F ( v_1, x_1; \ldots;  v_{i-1}, x_{i-1}; \;  \omega_V (v_i, x_i - x_{i+1})  
 v_{i+1}, x_{i+1}; \; v_{i+2}, x_{i+2};  
\nn
&& \qquad \qquad \qquad 
\ldots;  v_k, x_k;  v'_1, y_1; \ldots; v'_n, y_n; \epsilon ) \rangle   
\end{eqnarray*}
\begin{eqnarray*}
&& \qquad + 
\sum_{i=1}^n(-1)^i \; \langle w',  
\F \left( v_1, x_1; \ldots; v_k, x_k;  v'_1, y_1; \ldots; v'_{i-1}, y_{1-1};
\right.
\nn
&&\qquad \qquad \qquad  \left. 
  \omega_V (v'_i, y_i - y_{i+1}; ) \;  
 v'_{i+1}, y_{i+1}; v'_{i+2}, y_{i+2}; \ldots; v'_n, y_n; \epsilon \right) \rangle 
\end{eqnarray*}
\begin{eqnarray*}
&& \qquad +  \langle w', \omega_W \left(v_1, x_1 \right) \; 
\F (v_2, x_2; \ldots; v_k, x_k; v'_1, y_1; \ldots; v'_n, y_n; \epsilon) \rangle     
\end{eqnarray*}
\begin{eqnarray*}
 & & \qquad +  \langle w, (-1)^{k+n+1}    
 \omega_W(v'_{n+1}, y_{n+1})  
\; \F(v_1, x_1; \ldots; v_k, x_k; v'_1, y_1; \ldots; v'_n, y_n; \epsilon ) \rangle  
\end{eqnarray*}
\begin{eqnarray*}
&& \quad = 
\sum\limits_{u\in V} 
\langle w', \sum_{i=1}^k(-1)^i \; Y^W_{VW}(   
\F ( v_1, x_1; \ldots;  v_{i-1}, x_{i-1}; \; \omega_V (v_i, x_i- x_{i+1})  
 v_{i+1}, x_{i+1}; \;  
\nn
&& \qquad \qquad \qquad  
v_{i+2}, x_{i+2};  \ldots;  v_k, x_k), \zeta_1) u \rangle
 \langle w', Y^W_{VW}( \F(v'_1, y_1; \ldots; v'_n, y_n), \zeta_2) 
\overline{u} \rangle  
\end{eqnarray*}
\begin{eqnarray*}
&& \qquad + 
\sum\limits_{u\in V}  \sum_{i=1}^n(-1)^i \; \langle w',  Y^W_{VW}( 
\F \left( v_1, x_1; \ldots; v_k, x_k \right), \zeta_1) u \rangle 
\nn
&&
\qquad \qquad \qquad 
   \langle w',  Y^W_{VW}( \F(v'_1, y_1; \ldots; v'_{i-1}, y_{i-1};  
\nn
&& 
\qquad \qquad \qquad  \qquad \omega_V (v'_i, y_i - y_{i+1}; ) \;  
 v'_{i+1}, y_{i+1}; v'_{i+2}, y_{i+2}; \ldots; v'_n, y_n ), \zeta_2) \overline{u} \rangle  
\end{eqnarray*}
\begin{eqnarray*}
&& \qquad +  \sum\limits_{u\in V} \langle w', Y^W_{VW}(
 \omega_W \left(v_1, x_1 \right) \; \F (v_2, x_2 ; \ldots; v_k, x_k), \zeta_1) u \rangle  
\nn
&&
 \qquad \qquad
\langle w', Y^W_{VW}( \F( v'_1, y_1; \ldots; v'_n, y_n ), \zeta_2) \overline{u} \rangle    
\end{eqnarray*}
\begin{eqnarray*}
&&
\qquad +  \sum\limits_{u\in V}  \langle w', Y^W_{VW}( (-1)^{k+1}  
 \omega_W \left(v_{k+1}, x_{k+1} 
  \right) \;  \F (v_1, x_1 ; \ldots; v_k), \zeta_1) u\rangle  
\nn
&& \qquad \qquad \qquad 
\langle w', Y^W_{VW}(
\F( x_k; v'_1, y_1; \ldots; v'_n, y_n), \zeta_2) \overline{u} \rangle   
\end{eqnarray*}
\begin{eqnarray*}
&&
\qquad - 
\sum\limits_{u\in V} \langle w', (-1)^{k+1}   \langle w',  Y^W_{VW}(  
 \omega_W \left(v_{k+1}, x_{k+1} 
  \right) \; \F (v_1, x_1 ; \ldots; v_k, x_k), \zeta_1) u \rangle   
\nn
&&
\qquad \qquad  \qquad 
\langle w',  
Y^W_{VW}( \F(v'_1, y_1; \ldots; v'_{n}, y_{n}),  \zeta_2)  \overline{u} \rangle   
\nn
&&
\qquad +  \sum\limits_{u\in V}  \langle w', Y^W_{VW}(  
 \F(v_1, x_1; \ldots; v_k, x_k), \zeta_1)  u\rangle  
\nn
&&
\qquad \qquad 
 \langle w', Y^W_{VW}(  
\omega_W(v'_1, y_1)\; \F(v'_2, y_2; \ldots; v'_{n}, y_{n}  ), \zeta_2)  \overline{u}\rangle 
\nn
&& \qquad 
\sum\limits_{u\in V} \langle w', Y^W_{VW}(  
 \F(v_1, x_1; \ldots; v_k, x_k), \zeta_1) \rangle  
\nn
&& 
\qquad \qquad \langle w',  Y^W_{VW}(
   \omega_W(v'_1, y_1) \F( v'_2, y_2; \ldots; v'_n, y_n  ), \zeta_2) \rangle  
\end{eqnarray*} 
\begin{eqnarray*}
&& \quad = 
\sum\limits_{u\in V} 
\langle w',  \; Y^W_{VW}(  
\delta^k_m\F ( v_1, x_1; \ldots;  v_k, x_k), \zeta_1) u \rangle 
\nn
&& 
\qquad \qquad \qquad \qquad \langle w', Y^W_{VW}( \F(v'_1, y_1; \ldots; v'_n, y_n), \zeta_2) 
\overline{u} \rangle  
\nn
&& \qquad + (-1)^k 
\sum\limits_{u\in V}   \langle w',  Y^W_{VW}(
\F ( v_1, x_1; \ldots; v_k, x_k), \zeta_1) u \rangle   
\nn
&&
\qquad \qquad \qquad 
  \langle w',  Y^W_{VW}( \delta^n_{m'}  
\F(v'_1, y_1; \ldots;  v'_n, y_n ), \zeta_2 ) \overline{u} \rangle     
\end{eqnarray*}
\begin{eqnarray*}
&& \quad = 
\langle w',   
\delta^k_m \F ( v_1, x_1; \ldots;  v_k, x_k) \cdot_\epsilon 
\langle w',  \F(v'_1, y_1; \ldots; v'_n, y_n) \rangle   
\nn
&& \qquad + (-1)^k 
   \langle w',  
 \F ( v_1, x_1; \ldots; v_k, x_k) \cdot_\epsilon    
     \delta^n_{m'}  \F(v'_1, y_1; \ldots;  v'_n, y_n ) \rangle,  
\end{eqnarray*}
since, 
\begin{eqnarray*}
&& \sum\limits_{u\in V} \langle w', (-1)^{k+1}   Y^W_{VW}(  
 \omega_W \left(v_{k+1}, x_{k+1} 
  \right) \; \F (v_1, x_1 ; \ldots; v_k, x_k), \zeta_1) u \rangle   
\nn
&&
\qquad \qquad \langle w',  
Y^W_{VW}( \F(v'_1, y_1; \ldots; v'_n, y_n),  \zeta_2) \overline {u} \rangle   
\end{eqnarray*}
\begin{eqnarray*} 
&& 
=\sum\limits_{u\in V} \langle w', (-1)^{k+1} e^{\zeta_1 L_W{(-1)}}   Y_W(u, -\zeta_1)  \;  
 \omega_W \left(v_{k+1}, x_{k+1}  
  \right) \; \F (v_1, x_1 ; \ldots; v_k, x_k) \rangle   
\nn
&&
\qquad \qquad \langle w',  
Y^W_{VW}( \F(v'_1, y_1; \ldots; v'_n, y_n),  \zeta_2) \overline{u} \rangle   
\end{eqnarray*}
\begin{eqnarray*} 
&&
=\sum\limits_{u\in V} \langle w', (-1)^{k+1} e^{\zeta_1 L_W{(-1)}} 
 \omega_W \left(v_{k+1}, x_{k+1}  \right)  
 Y_W(u, -\zeta_1)  \;  
  \; \F (v_1, x_1 ; \ldots; v_{k}, x_k) \rangle  
\nn
&&
\qquad \qquad \langle w',  
Y^W_{VW}( \F(v'_1, y_1; \ldots; v'_{n}, y_{n}),  \zeta_2) \overline{u} \rangle   
\end{eqnarray*}
\begin{eqnarray*} 
\nn
&&
= 
\sum\limits_{u\in V} \langle w', (-1)^{k+1} \; 
\omega_W \left(v_{k+1}, x_{k+1} +\zeta_1  \right)\;  e^{\zeta_1 L_W{(-1)}}  
 Y_W(u, -\zeta_1)  \;  
  \; \F (v_1, x_1 ; \ldots; v_k, x_k) \rangle  
\nn
&&
\qquad \qquad \langle w',  
Y^W_{VW}( \F(v'_1, y_1; \ldots; v'_n, y_n),  \zeta_2) \overline{u} \rangle    
\end{eqnarray*}
\begin{eqnarray*} 
&&
=\sum\limits_{v\in V} 
\sum\limits_{u\in V} 
\langle v', (-1)^{k+1} \; \omega_W \left(v_{k+1}, x_{k+1}+\zeta_1  \right) w \rangle  
\nn
&&
\qquad \qquad \langle w',  e^{\zeta_1 L_W{(-1)}}  
 Y_W(u, -\zeta_1)  \;  
  \; \F (v_1, x_1 ; \ldots; v_k, x_k) \rangle  
\nn
&&
\qquad \qquad \langle w',  
Y^W_{VW}( \F(v'_1, y_1; \ldots; v'_n, y_n),  \zeta_2) \overline{u} \rangle   
\end{eqnarray*}
\begin{eqnarray*} 
&&
= 
\sum\limits_{u\in V} 
 \langle w',  e^{\zeta_1 L_W{(-1)}}   
 Y_W(u, -\zeta_1)  \;  
  \; \F (v_1, x_1 ; \ldots; v_k, x_k) \rangle  
\nn
&&
\qquad \qquad
 \sum\limits_{v\in V} \langle v', (-1)^{k+1} \; \omega_W \left(v_{k+1}, x_{k+1}+
\zeta_1  \right) w \rangle    
\langle w',  Y^W_{VW}( \F(v'_1, y_1; \ldots; v'_n, y_n),  \zeta_2) \overline{u} \rangle   
\end{eqnarray*}
\begin{eqnarray*} 
&&
= 
\sum\limits_{u\in V} 
 \langle w',  
 Y^W_{VW}(   
   \F (v_1, x_1 ; \ldots; v_k, x_k) , \zeta_1) u \; \rangle  
\nn
&&
\qquad \qquad 
 \langle w', (-1)^{k+1} \; \omega_W \left(v_{k+1}, x_{k+1}+\zeta_1  \right) 
\;  Y^W_{VW}( \F(v'_1, y_1; \ldots; v'_n, y_n),  \zeta_2) \overline{u} \rangle   
\end{eqnarray*}
\begin{eqnarray*} 
&&
= 
\sum\limits_{u\in V} 
 \langle w',   
 Y^W_{VW}(   
   \F (v_1, x_1 ; \ldots; v_k, x_k) , \zeta_1) u \; \rangle   
\nn
&&
\qquad \qquad 
 \langle w', (-1)^{k+1} \; \omega_W \left(v_{k+1}, x_{k+1}+\zeta_1  \right) 
\; e^{\zeta_2 L_W{(-1)}}    Y_W(\overline{u}, -\zeta_2) \; 
\F(v'_1, y_1; \ldots; v'_n, y_n) \rangle   
\end{eqnarray*}
\begin{eqnarray*} 
&&
= 
\sum\limits_{u\in V} 
 \langle w',  
 Y^W_{VW}(  
   \F (v_1, x_1 ; \ldots; v_k, x_k) , \zeta_1) u \; \rangle  
\nn
&&
\qquad 
 \langle w', (-1)^{k+1} \; 
\; e^{\zeta_2 L_W{(-1)}} \;   Y_W(\overline{u}, -\zeta_2)
\;
\omega_W \left(v_{k+1}, x_{k+1}+\zeta_1-\zeta_2 \right)  \; 
\F(v'_1, y_1; \ldots; v'_n, y_n)  \rangle    
\end{eqnarray*}
\begin{eqnarray*} 
&& 
\qquad \qquad 
= \sum\limits_{u\in V} \langle w', Y^W_{VW}(  
 \F(v_1, x_1; \ldots; v_k, x_k), \zeta_1) u \rangle  
\nn
&& 
\qquad \qquad \langle w',  Y^W_{VW}(
\omega_W(v'_1, y_1) \; \F( v'_2, y_2; \ldots; v'_n, y_n  ), \zeta_2) \overline{u} \rangle,    
\end{eqnarray*}
due to locality \eqref{porosyataw} of vertex operators, 
and arbitrariness of $v_{k+1}\in V$ and $x_{k+1}$,  
we can always put 
\[
\omega_W \left(v_{k+1}, x_{k+1}+\zeta_1-\zeta_2 \right)  =\omega_W(v'_1, y_1),  
\]
for $v_{k+1}=v'_1$, $x_{k+1}= y_1+\zeta_2-\zeta_1$.   
\end{proof}
Finally, we have the following 
\begin{corollary}
The multiplication \eqref{Z2n_pt_epsss0} extends the chain-cochain 
complex \eqref{conde}--\eqref{hat-complex} structure 
for all products $C^k_m(V, \W) \times C^n_{m'}(V, \W)$,   
$k$, $n \ge0$, $m$, $m' \ge0$.    
\hfill $\qed$
\end{corollary}
\begin{corollary}
The product \eqref{Z2n_pt_epsss0} and the product operator \eqref{hatdelta}  
endow the space $C^k_m(V, \W)$ $\times$ 
$C^n_m(V, \W)$, $k$, $n \ge0$, $m$, $m' \ge0$, 
 with the structure of a 
bi-graded differential algebra $\mathcal G(V, \W, \cdot, \delta^n_m)$. 
\hfill $\qed$
\end{corollary}
\section{Example: exceptional complex} 
\label{example}
In addition to the complex $(C^n_m(V, \W), \delta^n_m)$
 provided by \eqref{conde}--\eqref{hat-complex},  
 there exists an exceptional short complex  
$(C^2_{ex}(V, \W), \delta^2_{ex})$.  
 In \cite{Huang} we have  
\begin{lemma}
\label{lemmo}
 For $n=2$, there exists a subspace 
 $C^0_{ex}(V, \W)$  
\[
C_m^2(V, \W) \subset C^2_{ex}(V, \W) \subset C_0^2(V, \W),   
\]
 for all $m \ge 1$, with the action of coboundary operator 
$\delta^2_m$ defined.    
\hfill $\qed$
\end{lemma}
Let us recall some facts about the exceptional complex.  
 Consider the space $C_0^2(V, \W)$.  
 It consist of $\W_{z_1, z_2}$-elements with zero 
vertex operators composable. 
The space $C_0^2(V, \W)$ contains elements of $\W_{z_1, z_2}$ so that the action   
of $\delta_0^2$ is zero.  
Nevertheless, as for $\mathcal J^n_m(\Phi)$ in \eqref{Jnm0},
 Definition \ref{composabilitydef}, let us consider sum of   
projections 
\[
P_r: \W_{z_i, z_j} \to W_r, 
\]
for $r\in \C$, and $(i, j)=(1,2), (2, 3)$,   
so that the condition \eqref{Jnm0}  is satisfied for some elements  
similar to the action \eqref{Jnm0} of $\delta^2_0$. 
Separating the first two and the second two summands in \eqref{hatdelta},
 we find that for a subspace  
of $C_0^2(V, \W)$ (which we denote as $C_{ex}^2(V, \W)$),     
for $v_1$, $v_2$, $v_3 \in V$,   
 and arbitrary $w'\in W'$, $\zeta \in \C$, the following elements 
\begin{eqnarray}
\label{pervayaforma}
&& G_1(z_1, z_2,  z_3) 
\nn
&&
= \sum_{r\in \C} \Big( 
\langle w', E^{(1)}_W \left(  v_1, z_1;  
P_r \left(   
\F\left( v_2, z_2-\zeta;  v_3, z_3 - \zeta  \right)  \right) \rangle  \right.   
\nn
&&\quad +\langle w', \left. 
\F \left( v_1, z_1;  P_r \left( E^{(2)}_V 
\left( v_2, z_2-\zeta; v_3, z_3 -\zeta; \one_V \right),     
 \zeta \right) \right) \right.
\rangle \Big) 
\nn
&&
\nn
&& = \sum_{r\in \C} \big( \langle w', \omega_W \left( v_1, z_1 \right)\;        
 P_r \left( \F\left(v_2, z_2-\zeta; v_3, z_3 - \zeta \right) \right) \rangle  
\nn
&&\quad +\langle w', 
\F \left( v_1, z_1;  P_r \left(  \omega_V   
\left(v_2, z_2 -\zeta\right)  \omega_V \left( v_3, z_3 -\zeta \right) \one_V \right),       
 \zeta \right)
\rangle\big),  
\nn
\end{eqnarray}
and 
\begin{eqnarray}
\label{vtorayaforma}
&&
G_2(z_1, z_2, z_3) 
\nn 
 && \qquad =\sum_{r\in \C} \Big(  \langle w',  
\F \left(  
P_r \left(   E^{(2)}_V \left( v_1, z_1-\zeta, v_2, z_2 -\zeta; \one_V \right) \right),  
\zeta;  
 v_3, z_3 \right) \rangle  
\nn
&& \qquad \qquad+ \langle w',  
E^{W; (1)}_{WV} \left(P_r \left(    \F \left(   v_1,  z_1 -\zeta;   v_{2}, z_2 -\zeta \right), 
\zeta;    
 v_3, z_3 \right) \right) \rangle \Big) 
\nn
&&
\nn
&&
 \qquad = \sum_{r\in \C} \big(  \langle w', 
 \F \left(   P_r \left( \omega_V ( v_1, z_1 -\zeta ) \omega_V (v_2, z_2 -\zeta) \one_V, 
\zeta \right);   
 v_3, z_3 \right) \rangle    
\nn
&&\qquad \qquad+ \langle w', 
 \omega_V (v_3, z_3 ) \; P_r \left(\F \left(v_1, z_1-\zeta; v_2, z_2-\zeta \right)  \right) 
  \rangle \big),  
\end{eqnarray}
are absolutely convergent in the regions 
\[
|z_1-\zeta|>|z_2-\zeta|,   
\]
\[
 |z_2 -\zeta|>0,  
\] 
\[
|\zeta-z_3|>|z_1-\zeta|,  
\]
\[
|z_2-\zeta|>0, 
\]
where $z_i$, $1 \le i \le 3$. 
These functions can be analytically extended to  
rational form-valued functions in $z_1$ and $z_2$ with the only possible poles at 
$z_1$, $z_2=0$, and $z_1=z_2$.
Note that \eqref{pervayaforma} and \eqref{vtorayaforma} 
constitute the first two and the last two terms of  
\eqref{hatdelta} correspondingly.  
According to Proposition \ref{comp-assoc} (cf. Appendix \ref{composable}),  
 $C_m^2(V, \W)$ is a subspace of $C_{ex}^2(V, \W)$,  for $m\ge 0$, 
and $\F \in C_m^2(V, \W)$ are composable with $m$ vertex operators.  
Then we have 
\begin{definition}
\label{cobop}
 The coboundary operator 
\begin{equation}
\label{halfdelta}
\delta^2_{ex}: C_{ex}^2(V, \W)
\to C_0^3(V, \W), 
\end{equation}
is defined by   
\begin{eqnarray}
\label{ghalfdelta}
 \delta^2_{ex} \F &=& 
 \langle w', \omega_W (v_1, z_1 ) \;  
\F \left(v_2, z_2; v_3, z_3 \right) \rangle  
\nn
\quad \quad 
&-&
 \langle w',  
\F \left(  \omega_V( v_1, z_1)\; \; \omega_V (v_2, z_2 ) \one_V; v_3, z_3 \right)   
\rangle
\nn
 &&\quad
+ 
 \langle w', \F(v_1, z_1;  \; \omega_V (v_2, z_2)\;  \omega_V( v_3, z_3) \one_V) \rangle  
\nn
 &&\quad \quad 
+ 
 \langle w',  
 \omega_W (v_3, z_3) \; \F \left(v_1, z_1; v_2, z_2 \right)    
\rangle, 
\end{eqnarray}
for arbitrary $w'\in W'$, 
$\F \in C_{ex}^2(V, \W)$,  
$(v_1, v_2, v_3) \in V$ and $(z_1, z_2, z_3)\in F_3\C$.    
\end{definition}
In \cite{Huang} we also find  
\begin{proposition}
\label{cochainprop1}
The operator 
\eqref{ghalfdelta} provides the chain-cochain complex 
\[
\delta^2_{ex}\circ \delta^1_2=0,  
\]
\begin{equation}
\label{hat-complex-half}
0\longrightarrow C^0_3(V, \W)
\stackrel{\delta_3^0}{\longrightarrow} 
 C^1_2(V, \W)
\stackrel{\delta_2^1}{\longrightarrow} C_{ex}^2(V, \W)
\stackrel{\delta_{ex}^2}{\longrightarrow} 
 C_0^3(V, \W)\longrightarrow 0.
\end{equation}
\hfill $\qed$
\end{proposition}
Since  
\[
\delta_2^1 \;  C_2^1(V, \W) \subset  
 C_1^2(V, \W)\subset  
 C_{ex}^2(V, \W),
\]
the second formula follows from the first one, and 
\[
\delta^2_{ex}\circ  \delta^1_2 
= \delta^2_1 \circ \delta^1_2=0. 
\]
For elements of the spaces $C^2_{ex}(V, \W)$   
we have the following 
\begin{corollary}
The product of elements of the spaces $C^2_{ex} (V, \W)$ and  $C^n_m (V, \W)$ is given by 
\eqref{Z2n_pt_epsss0}, 
\begin{equation}
\label{pupa3}
\cdot_\epsilon: C^2_{ex} (V, \W) \times C^n_m (V, \W) \to C^{n+2}_m (V, \W),  
\end{equation}
and, in particular, 
\[
\cdot_\epsilon: C^2_{ex} (V, \W) \times C^2_{ex} (V, \W) \to C^4_0 (V, \W).   
\]
\end{corollary}
\begin{proof}
The fact that the number of formal parameters is $n+2$ in the product \eqref{Z2n_pt_epsss0} 
follows from  
Proposition \eqref{derga}.  
Consider the product \eqref{Z2n_pt_epsss0} for  
$C^2_{ex} (V, \W)$ and 
$C^n_m (V, \W)$.  
It is clear that, similar to considerations of the proof of Proposition \ref{tolsto}, 
the total number $m$ of vertex operators the product $\F$ is composable to remains the same. 
\end{proof}
\subsection{An application to cohomology invariants of vertex algebras}
\label{appi}
In this Subsection we give an explicit application of the product constructed in 
Section \ref{product} 
to cohomology invariants of vertex algebras by 
using the results of \cite{Ko}. 
Let us define a further product of a pair 
of elements of the spaces $C^k_m(V, \W)$ and $C^n_{m'}(V, \W)$,  
suitable for the formulation of cohomology invariants.  
Consider the forms 
$\F(v_1, z_1  ;   \ldots ;  v_n, z_k) \in   C_m^k(V, \W)$,    
$\widetilde{\F} (v_{k+1}, z_{k+1}; \ldots; v_{k+n}, z_{k+n}) \in   C_{m'}^n(V, \W)$,           
 with $r$ common vertex algebra elements (and, correspondingly, $r$ formal variables), and 
 $t$ common vertex operators mappings $\Phi$ and $\Psi$ are composable with. 
 Define an extra product  
of $\F$ and $\widetilde{\F}$,   
 $\F \cdot \widetilde{\F}: V^{\otimes(k +n-r)} \to  \W_{z_1, \ldots, z_{k+ n-r}}$,  
 $\F \cdot \widetilde{\F}$ $=$ $\left[\F,_{\cdot \epsilon} \widetilde{\F}\right]$     
$=$ $\F \cdot_\epsilon \widetilde{\F}$ $-$ $\widetilde{\F} \cdot_\epsilon \F$,     
where brackets denote the ordinary commutator on $\W_{z_1, \ldots, z_{k+ n-r}}$.  
Due to the properties of the maps $\F \in C_m^k(V, \W)$ and   
$\widetilde{\F} \in C_{m'}^n(V, \W)$,   
the product $\Phi \cdot \Psi$
belongs to the space $C_{m + m'- t }^{k +n-r}(V, \W)$.   
For $k=n$ and 
$\widetilde{\F} (v_{n+1}, z_{n+1}; \ldots;  v_{2n}, z_{2n})
= 
\F(v_1, z_1; \ldots;  v_n, z_n)$,     
and we obtain from \eqref{Z2n_pt_epsss0} that 
$\F(v_1, z_1; \ldots;  v_n, z_n)$ $\cdot$ 
 $\F(v_1, z_1; \ldots; v_n, z_n) =0$.   
Let us now introduce  
cohomology classes associated to grading-restricted vertex algebras.   
Usually, (e.g., \cite{G, CM, Ko}) cohomology classes  
 are introduced by means of   
an extra condition (in particular, the orthogonality condition) 
applied to differential forms, and leading to  
an integrability condition.  
 One calls a map $\F \in C_k^n(V, \W)$,      
closed if $\delta^n_k \F=0$.     
For $k \ge 1$, one calls it exact if there exists  
$\widetilde{\F} \in  C_{k-1}^{n+1}(V, \W)$, such that 
$\widetilde{\F}=\delta^n_k \F$.    
For $\F \in C^n_k(V, \W)$ we call the cohomology class of mappings 
 $\left[ \F \right]$ the set of all closed forms that differ from $\F$ by an  
exact mapping, i.e., for $\Lambda \in C^{n-1}_{k+1}(V, \W)$,  
$\left[ \F \right]= \F + \delta^{n-1}_{k+1} \Lambda$.    
In contrast to \cite{Ko}, our cohomology class is a functional of $v\in V$. 
That means that the actual functional form of $\F(v, z)$  
(and, therefore, $\langle w', \F \rangle$, for $w'\in W'$)       
 varies with various choices of $v\in V$. 
Under a natural extra condition, the double complexes \eqref{hat-complex} 
and \eqref{hat-complex-half}   
allow us to establish relations among elements of $C^n_m(V, \W)$-spaces.  
By analogy with the notion of the integrability for differential forms \cite{G}, 
 we use here the notion of the orthogonality for the spaces of a complex. 
For the complexes \eqref{hat-complex} and \eqref{hat-complex-half}  
let us require that for a pair of 
complex spaces $C_m^k(V, \W)$ and $C_{m'}^n(V, \W)$,  
 there exist subspaces 
$\widetilde{C}_m^k(V, \W)\subset C_m^k(V, \W)$, 
$\widetilde{C}_{m'}^n(V, \W)\subset C_{m'}^n(V, \W)$, 
 such that for all $\F \in \widetilde{C}_m^k(V, \W)$ and all 
$\widetilde{\F} \in \widetilde{C}_{m'}^n(V, \W)$,     
$\F \cdot \delta^n_{m'} \widetilde{\F}=0$.    
 Namely, $\F$ is supposed to be orthogonal  
to $\delta^n_{m'} \widetilde{\F}$ with respect to the new product.    
We call this the orthogonality condition for mappings of the double complexes 
\eqref{hat-complex} and \eqref{hat-complex-half}.
Note that in the case of differential forms considered on a smooth manifold, 
the Frobenius theorem for a distribution provides the orthogonality condition \cite{G}.   
The fact that both sides of the orthogonality condition  
 belong to the same double complex space, applies limitations 
to possible combinations of $(k, m)$ and $(n, m')$. 
Analogously to the proof given in \cite{Ko}, 
it turns out that the orthogonality condition   
applied to elements of complexes \eqref{hat-complex} and 
\eqref{hat-complex-half}  
gives non-vanishing cohomology classes 
$\left[\left(\delta^n_m  \F \right) \cdot \F  \right]$ 
independent on the choice of $\F \in C^n_m(V, \W)$, in particular, for 
$(n,m)=(1,2)$, $(0,3)$, $(1, t)$, $0 \le t \le 2$. 
\section*{Acknowledgment}
The author is supported by the Institute of Mathematics 
of the Academy of Sciences of the Czech Republic (RVO 67985840). 
I would like to thank two unanimous referees for very valuable comments and corrections. 
\section{Appendix: Grading-restricted vertex algebras and their modules}
\label{grading}
In this Section, following \cite{Huang},  
we recall basic properties of 
grading-restricted vertex algebras and their grading-restricted generalized 
modules. 
We work over the base field $\C$ of complex numbers. 
\begin{definition}
A vertex algebra   
$(V,Y_V,\mathbf{1}_V)$, (cf. \cite{K}),  consists of a $\Z$-graded complex vector space  
\[
V = \coprod_{n\in\Z}\,V_{(n)}, \quad \dim V_{(n)}<\infty, 
\]
 for each $n\in \Z$,   
and a linear map 
\[
Y_V:V\rightarrow {\rm End \;}(V)[[z,z^{-1}]], 
\]
 for a formal parameter $z$, and a  
distinguished vector $\mathbf{1}_V\in V$.    
The evaluation of $Y_V$ on $v\in V$ is the vertex operator 
\begin{equation}
\label{vop}
Y_V(v)\equiv Y_V(v,z) = \sum_{n\in\Z}v(n)z^{-n-1}, 
\end{equation}
with components 
$(Y_V(v))_n =v(n)\in {\rm End \;}(V)$, where $Y_V(v,z)\mathbf{1}_V = v+O(z)$. 
The vertex operators satisfy 
the vacuum axiom (the identity property)  
\[
Y_V ({\bf 1}_V , z) = {\rm Id}_V,
\]
 and 
 the locality axiom 
\[
(z - w)^N [Y_V (u_1, z), Y_V (u_2, w)] = 0, 
\]
 for some $N >> 0$,  all $u_1$, $u_2 \in V$. 
\end{definition}
\begin{definition}
\label{grares}
A grading-restricted vertex algebra satisfies 
the following conditions:
\begin{enumerate}
\item {Grading-restriction condition}:
$V_{(n)}$ is finite dimensional for all $n\in \Z$, and $V_{(n)}=0$ for $n\ll 0$; 

\item { Lower-truncation condition}:
For $u$, $v\in V$, $Y_V(u, z)v$ contains only finitely many negative  
power terms, that is, 
\[
Y_V(u, z)v\in V((z)), 
\] 
(the space of formal 
Laurent series in $z$ with coefficients in $V$);   
\item {Creation property}: For $u\in V$, 
\[
Y_V (u, z)\mathbf{1}_V\in V[[z]], 
\]  
and 
\[
\lim_{z\to 0}Y_V(u, z)\mathbf{1}_V= u; 
\]

\item {Duality}: 
For $u_1, u_2, v\in V$,  
\[
v'\in V'=\coprod_{n\in \mathbb{Z}}V_{(n)}',  
\]
where 
 $V_{(n)}'$ denotes 
the dual vector space to $V_{(n)}$ and $\langle\, . ,  .\rangle$ the evaluation 
pairing $V'\otimes V\to \C$, the series 
\begin{eqnarray}
\label{porosyata}
& & \langle v', Y_V(u_1, z_1)Y_V(u_2, z_2)v\rangle, 
\\
& & \langle v', Y_V(u_2, z_2)Y_V(u_1, z_1)v\rangle, 
\\
& & \langle v', Y_V (Y_V (u_1, z_1-z_2)u_2, z_2)v\rangle,  
\end{eqnarray}
are absolutely convergent
in the regions 
\[
|z_1|>|z_2|>0, 
\]
\[
|z_2|>|z_1|>0, 
\]
\[
|z_2|>|z_1 -z_2|>0, 
\]
 respectively, to a common rational function 
in $z_1$ and $z_2$ with the only possible poles at $z_1=0=z_2$ and   
$z_1=z_2$;  
\item {$L_V(0)$-bracket formula}: Let $L_V(0): V\to V$,  
be defined by 
\[
L_V(0)v=nv, \qquad n=\wt(v),  
\]
 for $v\in V_{(n)}$.  
Then
\begin{equation} 
\label{locomm}
[L_V(0), Y_V(v, z)]=Y_V(L_V(0)v, z)+z\frac{d}{dz}Y_V(v, z),  
\end{equation}
for $v\in V$. 
\item {$L_V(-1)$-derivative property}:  
Let 
\[
L_V(-1): V\to V,  
\]
 be the operator given by 
\[
L_V(-1)v=\res_z z^{-2}Y_V(v, z)\one_V=Y_{(-2)}(v) \one_V,  
\]
where $Y(-2)(v)$ is the $(-2)$-th vertex algebra mode of the 
vertex operator $Y(v, z)$  
for $v\in V$. Then for $v\in V$, 
\begin{equation}
\label{derprop}
\frac{d}{dz}Y_V(u, z)=Y_V(L_V(-1)u, z)=[L_V(-1), Y_V(u, z)]. 
\end{equation}
\end{enumerate}
\end{definition}
In addition to that,  we recall here the following definition (cf. \cite{BZF}): 
\begin{definition}
 A grading-restricted vertex algebra $V$ is called conformal of central 
charge $c \in \C$,
 if there exists a non-zero conformal vector (Virasoro vector) $\omega \in V_{(2)}$ such that the 
corresponding vertex operator 
\[
Y_V(\omega, z)=\sum_{n\in\Z}L_V(n)z^{-n-2}, 
\]
is determined by modes of Virasoro algebra $L_V(n): V\to V$ satisfying  
\[
[L_V(m), L_V(n)]=(m-n)L_V(m+n)+\frac{c}{12}(m^3-m)\delta_{m+b, 0}\; {\rm Id_V}. 
\]
\end{definition}
\subsection{Grading-restricted generalized $V$-module}
\begin{definition}
A {grading-restricted generalized $V$-module} is a vector space 
$W$ equipped with a vertex operator map 
\begin{eqnarray*}
Y_W: V\otimes W&\to& W[[z, z^{-1}]], 
\nn
u\otimes w&\mapsto & Y_W(u, w)\equiv Y_W(u, z)w=\sum_{n\in \Z}(Y_W)_n(u,w)z^{-n-1},  
\end{eqnarray*}
and linear operators $L_W(0)$ and $L_W(-1)$ on $W$ satisfying the following
conditions:
\begin{enumerate}
\item {Grading-restriction condition}:
The vector space $W$ is $\mathbb C$-graded, that is, 
\[
W=\coprod_{\alpha\in \mathbb{C}}W_{(\alpha)},
\]
 such that 
$W_{(\alpha)}=0$ when the real part of $\alpha$ is sufficiently negative; 
\item { Lower-truncation condition}:
For $u\in V$ and $w\in W$, $Y_W(u, z)w$ contains only finitely many negative 
power terms, that is, $Y_W(u, z)w\in W((z))$;  
\item {Identity property}:  
Let ${\rm Id}_W$ be the identity operator on $W$.  
Then 
\[
Y_W(\mathbf{1}_V, z)={\rm Id}_W;  
\] 
\item {Duality}: For $u_1$, $u_2 \in V$, $w\in W$,  
\[
w'\in W'=\coprod_{n\in \mathbb{Z}}W_{(n)}^*, 
\]
 $W'$ denotes 
the dual $V$-module to $W$ and $\langle\, .,. \rangle$ their evaluation 
pairing, the series 
\begin{eqnarray}
\label{porosyataw}
&& \langle w', Y_W(u_1, z_1)Y_W(u_2, z_2)w\rangle, 
\\ 
&& \langle w', Y_W(u_2, z_2)Y_W(u_1, z_1)w\rangle,  
\\
&& \langle w', Y_W(Y_V(u_1, z_1-z_2)u_2, z_2)w\rangle,  
\end{eqnarray}
are absolutely convergent
in the regions
\[ 
|z_1|>|z_2|>0,
\]
\[
|z_2|>|z_1|>0, 
\]
\[
|z_2|>|z_1-z_2|>0,
\]
 respectively, to a common rational function 
in $z_1$ and $z_2$ with the only possible poles at $z_1=0=z_2$ and  
$z_1=z_2$.  

\item { $L_W(0)$-bracket formula}: For  $v\in V$, 
\[
[L_W(0), Y_W(v, z)]=Y_W(L_V(0)v, z)+z\frac{d}{dz}Y_W(v, z); 
\]

\item { $L_W(0)$-grading property}: For $w\in W_{(\alpha)}$, there exists
$N \in \Z_+$ such that  
\begin{equation}
\label{gradprop}
(L_W(0)-\alpha)^Nw=0;  
\end{equation}

\item { $L_W(-1)$-derivative property}: For $v\in V$,
\begin{equation}
\label{derprop}
\frac{d}{dz}Y_W(u, z)=Y_W(L_V(-1)u, z)=[L_W(-1), Y_W(u, z)]. 
\end{equation}
\end{enumerate}
\end{definition}
The translation property of vertex operators 
\begin{equation}
\label{transl}
 Y_W(u, z) = e^{-z' L_W(-1)} Y_W(u, z+z') e^{z' L_W(-1)},  
\end{equation}
for $z' \in \C$, follows from 
 \eqref{derprop}.   
For $a\in \C$, 
 the conjugation property with respect to the grading operator $L_W{(0)}$ is given by 
\begin{equation}
\label{aprop}
 a^{ L_W{(0)} } \; Y_W(v,z) \; a^{-L_W{(0)} }= Y_W (a^{ L_W{(0)} } v, az).    
\end{equation} 
For $v\in V$, and $w \in W$, the intertwining operator 
\begin{eqnarray}
\label{interop}
&& Y_{WV}^W: V\to W,   
\nn
&&
v   \mapsto  Y_{WV}^W(w, z) v,     
\end{eqnarray}
 is defined by 
\begin{eqnarray}
\label{wprop}
Y_{WV}^W(w, z) v= e^{zL_W(-1)} Y_W(v, -z) w. 
\end{eqnarray}
\subsection{Group of automorphisms of formal parameters}
\label{perdozo}
Asume that $W$ is a quasi-conformal grading-restricted vertex algebra $V$-module.  
Let us recall some further facts from \cite{BZF} 
relating generators of Virasoro algebra with the group of  
automorphisms in complex dimension one. 
 Let us represent an element of ${\rm Aut}_z \; \Oo^{(1)}$ by the map   
\begin{equation}
\label{lempa}
z \mapsto \rho=\rho(z),
\end{equation}
given by the power series
\begin{equation}
\label{prostoryad}
\rho(z) = \sum\limits_{k \ge 1} a_k z^k, 
\end{equation}
$\rho(z)$ can be represented in an exponential form 
\begin{equation}
\label{rog}
f(z) = \exp \left(  \sum\limits_{k > -1} \beta_{k }\; z^{k+1} \partial_z \right) 
\left(\beta_0 \right)^{z \partial_z}.z, 
\end{equation}
where we express $\beta_k \in \mathbb C$, $k \ge 0$, 
through combinations of $a_k$, $k\ge 1$.  
 A representation of Virasoro algebra modes in terms 
of differential operators is given by \cite{K}  
\begin{equation}
\label{repro}
L_W(m) \mapsto - \zeta^{m+1}\partial_\zeta, 
\end{equation}
for $m \in \Z$. 
 By expanding \eqref{rog} and comparing to 
\eqref{prostoryad} we obtain a system of equations which,  
 can be solved recursively for all $\beta_k$. 
In \cite{BZF}, $v \in V$, they derive the formula  
\begin{eqnarray}
\label{infaction}
&&
\left[L_W(n), Y_W (v, z) \right] 
=  \sum_{m \geq -1}  
 \frac{1}{(m+1)!} \left(\partial^{m+1}_z z^{m+1}  \right)\;   
Y_W (L_V(m) v, z),    
\end{eqnarray}
of a Virasoro generator commutation with a vertex operator. 
Given a vector field 
\begin{equation}
\label{top}
\beta(z)\partial_z= \sum_{n \geq -1} \beta_n z^{n+1} \partial_z, 
\end{equation}
which belongs to local Lie algebra of ${\rm Aut}\; \Oo^{(1)}$, 
 one introduces the operator 
\[
\beta = - \sum_{n \geq -1} \beta_n L_W(n). 
\]
We conclude from \eqref{top} with the following 
\begin{lemma}
\begin{eqnarray}
\label{infaction000}
&&
\left[\beta, Y_W (v, z) \right] 
=  - \sum_{m \geq -1}  
 \frac{1}{(m+1)!} \left(\partial^{m+1}_z \beta(z)  \right)\;  
Y_W (L_V(m) v, z).   
\end{eqnarray}
\end{lemma}
Here we introduce the following definitions. 
\begin{definition}
We call a grading-restricted vertex algebra quasi-conformal if 
  it carries an action of ${\rm Der}\; \Oo^{(n)}$  
 such that commutation formula \eqref{infaction000}    
 holds for any 
$v \in  V$,  and $z=z_j$, $1 \le j \le n$, the element $L_V(-1) = - \partial_z$    
acts as the translation operator 
\[
L_V(0) = - z \partial_z,  
\]
 acts semi-simply with integral
eigenvalues, and the Lie subalgebra ${\rm Der}_+ \; \Oo^{(n)}$ acts locally nilpotently. 
\end{definition}
\begin{definition}
\label{primary}  
A vector $A$ which belongs to a quasi-conformal 
grading-restricted vertex algebra $V$ 
 is called 
primary of conformal dimension $\Delta(A) \in  \mathbb Z_+$ if  
\begin{eqnarray*}
L_V(k) A &=& 0,\;  k > 0, 
\nn
 L_W(0) A &=& \Delta(A) A. 
\end{eqnarray*}
\end{definition}
The formula \eqref{infaction000}  is used in \cite{BZF} 
in order to prove invariance of 
vertex operators multiplied by conformal weight differentials in case of primary states, and 
in generic case.  

Let us give some further definitions: 
\begin{definition}
A conformal grading-restricted vertex algebra  
is a conformal vertex algebra $V$, 
such that it module $W$ is equipped with an action of the 
Virasoro algebra and hence its Lie subalgebra ${\rm Der}_0 \; \Oo^{(n)}$
given by the Lie algebra of ${\rm Aut} \; \Oo^{(n)}$.  
\end{definition}
\begin{definition}
\label{quasiconf}
A grading-restricted vertex algebra $V$-module $W$ is called quasi-conformal if 
  it carries an action of local Lie algebra of ${\rm Aut}\; \Oo$  
such that commutation formula
 \eqref{infaction000}    
 holds for any 
$v \in  V$, the element $L_W(-1) = - \partial_z$, 
as the translation operator $T$,
\[
L_W(0) = - z \partial_{z},  
\]
  acts semi-simply with integral
eigenvalues, and the Lie subalgebra 
of the positive part of local Lie algebra of ${\rm Aut}\; \Oo^{(n)}$ 
 acts locally nilpotently. 
\end{definition}
Recall \cite{BZF} the exponential form $f(\zeta)$ \eqref{rog} 
of the coordinate transformation \eqref{lempa}  
$\rho(z) \in {\rm Aut}\; \Oo^{(1)}$. 
A quasi-conformal vertex algebra has the formula \eqref{infaction000}, thus  
it is possible 
by using the identification \eqref{repro}, to introduce the linear operator 
 representing $f(\zeta)$ \eqref{rog} on $\W_{z_1, \ldots, z_n}$,  
\begin{equation}
\label{poperator}
 P\left(f (\zeta) \right)= 
\exp \left( \sum\limits_{m >  0} (m+1)\; \beta_m \; L_V(m) \right) \beta_0^{L_W(0)},  
\end{equation}
 (note that we have a different normalization in it). 
In \cite{BZF} it was shown that   
the action of an operator similar to \eqref{poperator} 
on a vertex algebra element $v\in V_n$ contains finitely many terms, and 
subspaces 
\[
V_{\le m} = \bigoplus_{ n \ge K}^m V_n, 
\]
 are stable under all operators $P(f)$, $f \in  {\rm Aut}\; \Oo^{(1)}$.   
 In \cite{BZF} they proved the following 
\begin{lemma}
 The assignment
\[
 f \mapsto P(f), 
\]
 defines a representation of ${\rm Aut}\; \Oo^{(1)}$
on $V$, 
\[
 P(f_1 * f_2) = P(f_1) \; P(f_2),  
\]
 which is the inductive limit of the
representations $V_{\le m}$, $m\ge K$ with some $K$ for some product $*$ of $f_1$ and $f_2$.  
\end{lemma}
Similarly, \eqref{poperator} provides a representation operator on $\W_{z_1, \ldots, z_n}$. 
\subsection{Non-degenerate invariant bilinear form on $V$} 
In this Subsection we recall \cite{TZ} the notion of 
a non-degeneerate invariant bilinear form defined on $V$.  
\label{liza}
The subalgebra 
\[
\{L_V(-1),L_V(0),L_V(1)\}\cong SL(2,\mathbb{C}), 
\]
 associated with M\"{o}bius transformations on  
$z$ naturally acts on $V$,   (cf., e.g. \cite{K}). 
In particular, for $\lambda \in \C$, 
\begin{equation}
\gamma_{\lambda}=\left(
\begin{array}{cc}
0 & \lambda\\
-\lambda & 0\\	
\end{array}
\right)
:z\mapsto w=-\frac{\lambda^{2}}{z},
 \label{eq: gam_lam}
\end{equation}
is generated by 
\[
T_\lambda= \exp\left(\lambda L_V{(-1)}\right)  
\; \exp\left(\lambda^{-1}L_V(1)\right) \; \exp\left(\lambda L_V(-1)\right),   
\]
 where
\begin{equation}
T_\lambda Y(u,z)T_\lambda^{-1}=  
Y\left(\exp \left(-\frac{z}{\lambda^{2}}L_V(1)\right)
\left(-\frac{z}{\lambda}\right)^{-2L_V(0)}u,-\frac{\lambda^{2}}{z}\right).  
\label{eq: Y_U}
\end{equation}
In our considerations (cf. Appendix \ref{sphere}) of the Riemann sphere  
sewing, we use in particular,  
the M\"{o}bius map 
\[
z\mapsto z'= \epsilon/z,
\] 
 associated with the sewing condition \eqref{pinch} with 
\begin{equation}
\lambda=-\xi\epsilon^{\frac{1}{2}},
\label{eq:lamb_eps}
\end{equation}  
and $\xi\in\{\pm \sqrt{-1}\}$. 
The adjoint vertex operator \cite{K, FHL}    
is defined by    
\begin{equation}
Y^\dagger(u,z)=\sum_{n\in \Z}u^\dagger (n)z^{-n-1}  
= T_\lambda Y(u,z)T_\lambda^{-1}. 
\label{eq: adj op}
\end{equation}
A bilinear form $\langle . , . \rangle_{\lambda}$ on $V$ is 
invariant if for all $a$, $b$, $u\in V$,  
\begin{equation}
\langle Y(u,z)a,b\rangle_{\lambda} = 
\langle a,Y^\dagger(u,z)b\rangle_{\lambda},  
\label{eq: inv bil form}
\end{equation}
i.e., 
\[
 \langle u(n)a,b\rangle_{\lambda} = 
\langle a,u^\dagger(n)b\rangle_\lambda. 
\]
Thus, it follows that 
\begin{equation}
\label{dubay}
\langle L_V(0)a,b\rangle_{\lambda} =\langle a,L_V(0)b\rangle_\lambda, 
\end{equation}
 so that 
\begin{equation}
\label{condip}
\langle a,b\rangle_{\lambda} =0, 
\end{equation}
  if $wt(a)\not=wt(b)$ for homogeneous $a$, $b$. 
 One also finds 
\[
\langle a,b\rangle_\lambda = \langle b,a \rangle_\lambda, 
\]
 and it is non-degenerate if and only if $V$ is simple.  
 Given any $V$ basis $\{ u^\alpha \}$ we define the 
dual $V$ basis $\{ \overline{u}^{\beta}\}$ where  
$\langle u^\alpha ,\overline{u}^\beta\rangle_\lambda=\delta^{\alpha\beta}$.  
\section{Appendix: A sphere formed from sewing of two spheres} 
\label{sphere}
 In this appendix we recall some facts from \cite{TZ}. 
The matrix element for a number of vertex operators of a vertex algebra
 is usually associated \cite{FHL, FMS, TUY}  
with a vertex algebra character on a sphere. 
We extrapolate this notion to the case of $\W_{z_1, \ldots, z_n}$ spaces.  
In Section \ref{product} we explained that a space $\W_{z_1, \ldots, z_n}$
can be associated with a Riemann sphere with marked points. 
 The product of two such spaces is then associated 
with a sewing of two such spheres with a number of marked  
points 
and extra points with local coordinates identified with formal parameters of 
$\W_{x_1, \ldots, x_k}$ and $\W_{y_1, \ldots, y_n}$.  
In order to supply an appropriate geometric construction for the product,  
 we use the $\epsilon$-sewing procedure (described in this Appendix) for two initial spheres
 to obtain a matrix element associated with \eqref{gendef}. 
\begin{remark}
In addition to the $\epsilon$-sewing procedure of 
two initial spheres, one can alternatively use  
the self-sewing procedure \cite{Y} of the sphere to get the torus.  Then by 
sending parameters 
to an appropriate limit by shrinking the genus to zero. 
As a result, one obtains again the sphere but with a  
different parameterization. 
In this paper we focus on the  
$\epsilon$-formalizm only.   
\end{remark}
In our particular case of $\W$-valued rational functions  
obtained from matrix elements \eqref{deff} we take two 
Riemann spheres $\Sigma^{(0)}_a$, $a=1$, $2$, 
as two initial auxiliary  
 spaces. 
The resulting 
space is formed by 
the sphere $\Sigma^{(0)}$ obtained by the sewing procedure of $\Sigma^{(0)}_a$.  
The formal parameters $(x_1, \ldots, x_k)$ and $(y_1, \ldots, y_n)$ are identified with  
local coordinates of $k$ and $n$ points 
on two initial spheres $\Sigma^{(0)}_a$, $a=1$, $2$ correspondingly.   
In the $\epsilon$ sewing procedure, some $r$ points parameters  
among 
$(p_1, \ldots, p_k)$ 
may coincide with 
points 
among $(p'_1, \ldots, p'_n)$ 
when we identify the annuluses \eqref{zhopki} below.    
This corresponds to the singular case of coincidence of $r$ formal parameters. 

Consider the sphere formed by sewing together 
two initial spheres in the sewing scheme referred to  
as the $\epsilon$-formalism in \cite{Y}.  
Let $\Sigma_a^{(0)}$,   
$a=1$, $2$ 
be two initial spheres.  
Introduce a complex sewing
parameter $\epsilon$ where 
\[
|\epsilon |\leq r_1r_2, 
\]
Consider $k$ distinct points on $p_i \in  \Sigma_1^{(0)}$, $i=1, \ldots, k$, 
with local coordinates $(x_1, \ldots, x_k) \in F_k\C$,   
and distinct points $p_j \in  \Sigma_2^{(0)}$, $j=1, \ldots, n$, 
with local coordinates $(y_1, \ldots , y_n)\in F_n \C$,   
with
\[
\left\vert x_i\right\vert 
\geq |\epsilon |/r_2, 
\]
\[
\left\vert y_i \right\vert \geq |\epsilon |/r_1.  
\] 
Choose a local coordinate $z_a \in \mathbb{C}$  
on $\Sigma^{(0)}_a$ in the
neighborhood of points $p_{a}\in\Sigma^{(0)}_a$, $a=1$, $2$. 
Consider the closed disks 
$\left\vert \zeta_a \right\vert \leq r_a$,    
 and excise the disk 
\begin{equation}
\label{disk}
\{
\zeta_a, \; \left\vert \zeta_a\right\vert \leq |\epsilon |/r_{\overline{a}}\}\subset 
\Sigma^{(0)}_a, 
\end{equation}
%
to form a punctured sphere  
\begin{equation*}
\widehat{\Sigma}^{(0)}_a=\Sigma^{(0)}_a \backslash \{\zeta_a,\left\vert  
\zeta_a\right\vert \leq |\epsilon |/r_{\overline{a}}\}.
\end{equation*}
We use the convention 
\begin{equation}
\overline{1}=2,\quad \overline{2}=1.  
\label{bardef}
\end{equation}
Define the annulus
\begin{equation}
\label{zhopki}
\mathcal{A}_a=\left\{\zeta_a,|\epsilon |/r_{\overline{a}}\leq \left\vert 
\zeta_a\right\vert \leq r_a\right\}\subset \widehat{\Sigma}^{(0)}_a,
\end{equation}
and identify $\mathcal{A}_1$ and $\mathcal{A}_2$ as a single region  
$\mathcal{A}=\mathcal{A}_1\simeq \mathcal{A}_2$ via the sewing relation  
\begin{equation}
\label{pinch}
\zeta_1\zeta_2=\epsilon.    
\end{equation}
In this way we obtain a genus zero compact Riemann surface 
\[
\Sigma^{(0)}=\left\{ \widehat {\Sigma}^{(0)}_1
\backslash \mathcal{A}_1 \right\} 
\cup \left\{\widehat{\Sigma}^{(0)}_2 \backslash  
\mathcal{A}_2\right\}\cup \mathcal{A}. 
\] 
This the sphere form a suitable geometric model 
for the construction of a product of $\W$-valued rational forms 
in Section \ref{product}.  
\section{Appendix: a proof of Proposition \ref{ndimwinv}}  
\label{proof}
In this Appendix we give a proof of 
Proposition \ref{nezc},  
namely, we prove that Definition \ref{wspace} 
is independent of 
the choice of formal parameters. 
Let us first recall definitions required for that. 
Let
 \[
{\rm Aut}_{z_1, \ldots, z_n} \; \Oo^{(n)} ={\rm Aut}_{\mathbb{C}}[[z_1,...,z_n]],   
\]
be the group of formal 
automorphisms 
of $n$-dimensional formal power series algebra ${\mathbb{C}}[[z_1, \ldots, z_n]]$. 
 
Let $W$ be a quasi-conformal module for a grading-restricted vertex algebra $V$. 
The $\mathbb Z$-grading on $W$ is bounded from below, 
\[ 
 W = \bigoplus_{k > k_0 }  W_k,  
\] 
for some $k_0 \in \Z$.   
Since the vector fields $z^{k+1} \partial_z$ with $k\in{\mathbb N}$  
acts on $W$ as the operators $-L_W(k)$ of degree $-k$, 
the action of the Lie subalgebra 
${\rm Der}_+ \;\Oo^{(n)}$ is locally nilpotent.  
 Furthermore, $z\partial_z$ acts as the grading operator $L_W(0)$, which 
is diagonalizable with integral eigenvalues. 
Thus, the action of ${\rm Der} \;\Oo^{(n)}$ 
on a conformal vertex algebra $V$ can be exponentiated  
to an action of ${\rm Aut}_{z_1, \ldots, z_n}  \; \Oo^{(n)}$. 

We write an element of ${\rm Aut}_{z_1, \ldots, z_n} \; \Oo^{(n)}$ as  
\begin{eqnarray*}
 (z_1, \ldots, z_n) \to \rho&=&(\rho_1, \ldots, \rho_n), 
\nn
\rho_i&=&\rho_i(z_1, \ldots, z_n),  
\end{eqnarray*} 
for $i=1, \ldots, n$, 
 where $\rho_i$ are defined by elements of $\mathfrak{m}\in\Oo^{(n)}$ 
\begin{eqnarray}
\label{prostoryad}
\rho_i(z_1, \ldots, z_n) = 
\sum\limits_{{i_1\geq 0, \ldots, i_n\geq 0, \atop \sum_{j=1}^ni_j \geq 1}} 
a_{(i_1, \ldots,  i_k)} z_1^{i_1} \ldots z_k^{i_k}, 
 \;\; a_{(i_1, \ldots, i_k)}\in{\mathbb C},  
\end{eqnarray}
and the images of $\rho_i$, $i=1, \ldots, n$, in the finite dimensional ${\mathbb C}$-vector 
space $\mathfrak{m}/\mathfrak{m}^2$ are linearly independent.
Let us denote 
\begin{eqnarray*}
{\bf v} &=& \left( v_1 \otimes \ldots \otimes   v_n\right),  
\\
{\bf z} &=& \left( z_1, \ldots, z_n\right), 
\\
w_i &=&\rho_i(z_1, \ldots, z_n), 
\\ 
{\bf w} &=& (w_1, \ldots, w_n).  
\end{eqnarray*}
The natural object that turns to be invariant with respect to the action of the group 
${\rm Aut}_{z_1, \ldots, z_n} \; \Oo^{(n)}$ 
is given by the matrix element 
of the $n$-vector 
\begin{equation}
\label{overphi}
\langle w', \overline{\Phi} \left({\bf v}, {\bf z} \;  {\bf dz} \right)\rangle  = \langle w',
\left[ 
\Phi \left( v_1, z_1 \; dz_{i(1)};  \ldots;    
v_n, z_n \;  dz_{i(n)}\right)  \right] \rangle,    
\end{equation}
containing $n$ $\F$-entries,     
where $i(j)$ denotes the cyclic permutation of $(1, \ldots, n)$ starting with $j$.  
 In this paper we use \eqref{bomba} which is related to \eqref{overphi}. 
Due to \eqref{loconj}, \eqref{overphi} can be written in the form  
\begin{eqnarray*}
\langle w', \overline{\Phi} \left( {\bf dz} \; {\bf v} \right)
 \left( {\bf z} \right) \rangle   
&=& \langle w', 
\left[  
\left( dz_{i(J)} \right)^{-L(0)_W} \Phi \left(
 \left( \left(  dz_{i(J)} \right)^{L_0^{(V)}} 
{\bf v} \right) \left( {\bf z} \right)  \right) 
\right] \rangle  
\nn
&=& 
\langle w',
\left[ 
\left( dz_{i(J)} \right)^{-L(0)_W} \Phi \left( \left(dz_{i(J)} \right)^{\wt(v_{J})} 
{\bf v} \left({\bf z} \right) \right)
\right] \rangle,
\end{eqnarray*}
coherent with the one-dimensional case of \cite{BZF} and containing $\wt(v)$-differentials. 
The idea to use  torsors \cite{BZF} is to represent the action of $\rho$ \eqref{prostoryad} 
 of the group ${\rm Aut}_{z_1, \ldots, z_n}\; \Oo^{(n)}$ 
 on formal parameters ${\bf z}$ in vectors $({\bf v}, {\bf z})$ to 
the action by $V$-operators on vertex algebra states ${\bf v}$.  
Recall the standard representation of the Virasoro mode \cite{K}  
\[
z_j^{m+1} \partial_{z_j}      
\mapsto -  L_m, \;\; m \in \Z. 
\]
In order to represent the action of the group 
${\rm Aut}_{z_1, \ldots, z_n} \; \Oo^{(n)}$ on the variables 
 $(z_1, \ldots, z_n)$ of 
 $\overline{\F}$ \eqref{overphi} on $(v_1, \ldots, v_n)$, 
we have to transfer (as in $n=1$ case of \cite{BZF}) 
 to an exponential form of \eqref{prostoryad}. 
The coefficients $\beta^{(j)}_{r_1, \ldots, r_n} \in \C$ are recursively   
found \cite{GR} in terms of coefficients $a^{(i)}_{r_1, \ldots, r_n}$ of \eqref{prostoryad}.  
We introduce the linear operators   
\[
R (\rho_1, \ldots, \rho_n):  V^{\otimes n} \;\to \; V^{\otimes n},     
\] 
and define the action  
\begin{equation}
\label{act}
\overline{ \Phi} \left( {\bf v}, {\bf w} \; {\bf dw } \right) = 
{\rm R} (\rho_1, \ldots, \rho_n)\; \overline{ \Phi }  
 \left(  {\bf v}, {\bf z} \; {\bf dw} \right). 
\end{equation}
\begin{proof}
Consider the vector  
\begin{equation}
\label{overphi}
\overline{\Phi}= \left[  
\Phi \left( 
v_1, w_1 \; dw_{{\it i}(1)}; \ldots;  
v_n,  w_n\; dw_{i(n)} \right)  \right],  
\end{equation}
with primary $(v_1, \ldots, v_n)$.  
Note that 
\begin{equation}
\label{shifty}
dw_j=\sum\limits_{i=1}^n \frac{\partial \rho_j}{\partial z_i } dz_i, 
\quad \partial_{z_i} \rho_j =\frac{\partial \rho_j}{\partial z_i },  
\end{equation}  
(as in \cite{BZF}, we skip the complex conjugated part $d\overline{z}_i$). 
By definition \eqref{act} of the action of ${\rm Aut}_{z_1, \ldots, z_n}\; \Oo^{(n)}$, and  
due to \eqref{shifty} by rewriting $dw_i$,
we have  
\begin{eqnarray*}
&&  \overline{\Phi} ( 
{\bf v}, {\bf w} \; {\bf dw})   
=  {\rm R}(\rho_1, \ldots, \rho_n) \;  
\left[  {\Phi} \left( v_1, z_1 \; dw_{i(1)};  \ldots;   
v_n, 
 z_n \; dw_{i(n)}  \right)
\right]
\\
 && \qquad =  {\rm R}(\rho_1, \ldots, \rho_n) \; 
\\
&&
\qquad  \qquad\left[  {\Phi} \left( v_1, z_1 \;  
\sum\limits_{j=1}^n \partial_j \rho_{i(1)} \; dz_j; \ldots;  
v_n,  z_n \; \sum\limits_{j=1}^n \partial_j \rho_{i(n)} \; dz_j      
\right) 
\right]. 
\end{eqnarray*}
 By using \eqref{lder1} and linearity of the mapping $\Phi$,  
 we obtain from the last equation 
\begin{equation}
\label{norma}
\overline{\Phi} \left({\bf v}, {\bf w} \;{\bf dw}\right)  
=
 \left[
 {\Phi} \left(  
v_1, z_1 \; dz_{i(1)}; \ldots;  
v_n, z_n  \;dz_{i(n)} \right)    
\right], 
\end{equation}
with  
\begin{equation}
\label{rovno} 
{\rm R} (\rho_1, \ldots, \rho_n)= 
\left[ \widehat \partial_{J} \rho_{i(I)}\right]
=
\left[
\begin{array}{c}
\widehat \partial_{J} \rho_{i_1(I)}
\\
\widehat \partial_{J } \rho_{i_2(I)} 
\\
\ldots  
\\
\widehat \partial_{J} \rho_{i_n(I)}    
\end{array}
\right]. 
\end{equation}
The index operator $J$ takes the value of index $z_j$ of arguments in the vector \eqref{norma}, 
while the index operator $I$ takes values of index of differentials $dz_i$ in each entry of 
the vector $\overline{\Phi}$ \eqref{overphi}. 
 Thus, the index operator 
$i(I)=(i_I, \ldots, i_n(I))$       
is given by consequent cycling permutations of $I$. 
Taking into account the property \eqref{lder1},  
we define the operator 
\begin{equation}
\label{hatrho}
\widehat\partial_J \rho_a =   
\exp\left(   - \sum\limits_{ {(r_1 \ldots r_n), \quad 
\sum\limits_{i=1}^n r_i  \ge 1 }}  
r_J\; \beta^{(a)}_{r_1, \ldots, r_n}\; \zeta^{r_1}_1 \;  \ldots \; 
 \zeta^{r_J}_J \ldots \zeta^{r_n}_n \;  L_{(W)}(-1)  \right), 
\end{equation}
which contains index operators $J$ as index of a  
dummy variable $\zeta_J$ turning into $z_j$, $j=1, \ldots, n$.  
\eqref{hatrho} acts on each argument of mappins $\Phi$ 
in the vector $\overline{\Phi}$ \eqref{overphi}.   
Due to the definition of a grading-restricted vertex algebra, 
the action of operators $R\left(\rho_1, \ldots , \rho_n\right)$ for
$i=1, \ldots, n$, on $v \in V$ results in a sum of finitely many terms.
Similar to \cite{BZF}, for $n=1$, one proves  
\begin{lemma}
The mappings   
$\rho(z_1, \ldots, z_n ) \mapsto R
\left(\rho_1, \ldots, \rho_n \right)$,   
for $i$, $j=1, \ldots, n$, 
 define a representation of ${\rm Aut}_{z_1, \ldots, z_n} \; \Oo^{(n)}$ 
on $V^{\otimes^n}$ by  
\[
{\rm R}
\left(\rho \circ \rho'\right) = {\rm R} 
\left(\rho\right) \; {\rm R} 
\left(\rho'\right),  
\]
for $\rho$, $\rho' \in {\rm Aut}_{z_1, \ldots, z_n} \; \Oo^{(n)}$.
\end{lemma}  
We then conclude with 
\begin{eqnarray}
&&
 \left[ \overline{\Phi} (  
  v_1, z_1\; dz_1;    \ldots;  
   v_n, \ldots, z_n\; dz_n)\right]. 
\end{eqnarray}
Thus the vector $\overline{\Phi}$ \eqref{overphi} is invariant, 
i.e., 
\begin{equation}
\label{phiproperty}
 \overline{\Phi} \left( 
{\bf v}, {\bf w}\; {\bf dw} \right) =  \overline{\Phi} \left( 
{\bf v}, {\bf z} \;{\bf dz}\right).     
\end{equation}

Recall that the construction of the spaces $C^n_m(V, \W)$ of the complex.    
It assumes 
that $\Phi \in C^n_m(V, \W)$ is composable with $m$  
vertex operators. 
In one-dimensional complex case,
 \cite{BZF} they proved that 
 a vertex operator multiplied to the $\wt(v_i)$-power of the differential 
$Y_W (v_i, z_i)\; dz_i^{\wt(v_i)}$    
is invariant with respect to the action of the group ${\rm Aut}_{z_i}\;\Oo^{(1)}$. 
Here we prove that $Y_W (v_i, z_i) dz_i^{\wt(v_i)}$  
is invariant with respect to the change of the local coordinates  
$z_i\mapsto w_i(z_1, \ldots, z_n)$.

 Let $(z_1, \ldots, z_n)$ be an open ball $D^{(n)}_{\bf z_0}$ of local formal coordinates 
around a fixed-value ${\bf z_0}$ of $(z_1, \ldots, z_n)$.   
Define a $\wt(v_i)$-differential 
on $D^{(n)}_{\bf z_0}$ with values in $End \; (W)_{\bf z_0}$ as follows:    
identify $End \; ( W)_{\bf z_0}$ with $End \;\;  W$ 
using the formal parameters $(z_1, \ldots, z_n)$, and set
\[
\omega_{i,x} =  Y_W \left(v_i, z_i \right) \; dz_i^{\wt(v_i)} .     
\] 
Let 
$(w_1, \ldots, w_n) = \left(\rho_1(z_1, \ldots, z_n), \ldots, \rho_n(z_1, \ldots, z_n)\right)$, 
be another set of formal parameters on an $n$-dimensional ball $D^{(n)}_{\bf z_0}$.  
Let us express the set of $\wt(v_i)$-differentials on $D^{(n), \times}_{\bf z_0}$ 
\[
Y_W (v_i, w_i) \; dw_i^{\wt(v_i)}, 
\]
 $i=1, \ldots, n$, in terms of 
 of the parameters $(z_1, \ldots, z_n)$.  
We would like to show that it coincides with the set of $\wt(v_i)$-differentials  
$Y_W(v_i, w_i)\; dz_i^{\wt(v_i)}$. 

Recall the notion of torsors (Section \ref{grading}). 
Consider a vector $(v_i, z_1, \ldots, z_n) \in W_{\bf z_0}$ with $v_i \in  V$.       
Then the same vector equals 
\[
\left(R_i^{-1}\left (\rho_1, \ldots, \rho_n\right) v_i, w_1, \ldots, w_n \right),
\] 
i.e., it is identified with
\[
R_i^{-1}\left(\rho_1, \ldots, \rho_n\right) v_i \in V,  
\]
using the formal parameters $(w_1, \ldots, w_n)$.  
Here $R_i\left(\rho_1, \ldots, \rho_n\right)$ 
is an operator representing transformation of $z_i \to w_i$,  
as an action on $V$. 
Therefore if we have an operator on $W_{\bf z_0}$ which  
is equal to a ${\rm Aut} \; \Oo^{(n)}$-torsor $S$ 
under the identification $End \; \;  W_{\bf w_0} \in End \; \;  W$  
using the formal parameters    
$(w_1, \ldots, w_n)$,
 then this operator equals
\[
R_i\left(\rho_1, \ldots, \rho_n\right) \;S \;R_i^{-1} \left(\rho_1, \ldots, \rho_n \right), 
\] 
under the identification  
$End \; \;   W_{\bf z_0} \in End \; \;  W^{(i)}$ using  
the combined parameters $\left(v_i, z_1, \ldots, z_n\right)$.  
Thus, in terms of $(v_i, z_1, \ldots, z_n)$, 
 the differential 
 $Y_W (v_i, w_i) \; dw_i^{\wt(v_i)}$    
becomes
\[
Y_W(v_i, z_i)\; dz_i^{\wt(v_i)} =R_i(\rho) \;  
Y_W \left(v_i, \rho(z_1, \ldots, z_n) \right)  \;  
R_i^{-1} (\rho) \;  
 dw_i^{\wt(v_i)}.
\] 
According to Definition \eqref{initialspace}, 
elements $\Phi$ are composable with $m$ vertex operators.  
Thus we see that \eqref{overphi} is a canonical object of $C^n_m(V, \W)$. 
We have proved that elements of the spaces ${C}^{n}_{m}(V, \W)$ 
are  independent 
on the choice of formal parameters.   
\end{proof}

\end{document}